\documentclass[a4paper,10pt]{article}

\usepackage[utf8]{inputenc}
\usepackage{amsmath, amssymb, amsthm}
\usepackage{xcolor}
\usepackage{graphicx} 	
\usepackage{enumitem}
\usepackage{esint} 		
\usepackage{geometry}
\usepackage{hyperref}

\newcommand{\dd}{\, \mathrm{d}}
\newcommand{\eps}{\varepsilon}
\newcommand{\hau}{d\mathcal{H}^2}
\newcommand{\N}{\mathbb{N}}
\newcommand{\R}{\mathbb{R}}
\usepackage[normalem]{ulem}	
\renewcommand{\vec}{\boldsymbol}    

\newtheorem{theorem}{Theorem}
\newtheorem{lemma}{Lemma}
\newtheorem{proposition}{Proposition}

\newtheorem*{remark*}{Remark}

\newcommand{\lf}{\liminf_{j\to \infty}}

\newcommand{\intp}{\fint_{x^2+y^2<1}}
\newcommand{\vareg}{\widehat{\varepsilon}}
\usepackage{authblk}

\title{Dimension reduction for gradient damage models in slender rods}

\author{E.~Bonnetier\thanks{Institut Fourier, Universit\'e Grenoble-Alpes, France. Email: \texttt{Eric.Bonnetier@univ-grenoble-alpes.fr}.} 
        \hspace{1cm} D.~Henao\thanks{Instituto de Ciencias de la Ingenier{\'{\i}}a, Universidad de O'Higgins, Rancagua, Chile. Email: \texttt{duvan.henao@uoh.cl}.}
        \hspace{1cm} V.~Ramos\thanks{Departamento de Ingenier{\'{\i}}a Matem\'atica, Universidad de Chile, Chile.  Email:\texttt{vramos@dim.uchile.cl}.}
}

\date{} 

\begin{document}
\parindent 0cm

\maketitle

\begin{abstract}
This paper presents a method for reducing a three-dimensional gradient damage model 
to a one-dimensional model for slender rods (with a small radius-to-length ratio, 
$\delta = R/L \to 0$). The 3D model minimizes an energy functional that includes elastic 
strain energy, a damage-dependent degradation function $a_\eta(\alpha)$, a damage energy 
term $w(\alpha)$, and a gradient term penalizing abrupt damage variations. 
After non-dimensionalizing and rescaling, the problem is reformulated on a unit cylinder, 
and the behaviour of the energy functional is analyzed as $\delta$ approaches zero. 
Using $\Gamma$-convergence, we show that the sequence of 3D energy functionals 
converges to a~1D functional, defined over displacement and damage fields that are 
independent of transverse coordinates. Compactness results guarantee the weak convergence 
of strains and damage gradients, while lower and upper bound inequalities confirm the energy limit. Minimizers of the 3D energy are proven to converge to the minimizers of the 1D energy, with strains approaching a diagonal form indicative of uniaxial deformation. 
\medskip

\textbf{Keywords:} Gradient damage model, Dimensional reduction, 
$\Gamma$-convergence, Energy functional, Uniaxial deformation.
\end{abstract}


\section{Introduction}

Gradient damage models (see, e.g., \cite{marigo2016overview,ambrosio1990approximation,ambrosio1992approximation,braides1998approximation,bourdin2000numerical,bourdin2008variational,tanne2018crack,kumar2020revisiting,bourdin2025}, 
and the references therein) have been used to describe the nucleation and propagation of cracks 
in brittle and quasi-brittle materials, in very challenging problems such as
the formation of regularly spaced cracks with very complex geometries in thin films subject to thermal shocks \cite{bourdin2014morphogenesis,sicsic2014initiation},
or such as determining the causes of the damage observed in the ashlar masonry work 
of the French Panth\'eon \cite{lancioni2009variational}.
In these models, cracks correspond to localized bands, where an internal damage variable is activated
that reduces the stiffness of the material.
A first stage consists in the onset of damage from an initially elastic material, 
when the stress reaches a well-defined intrinsic limit, 
which can be identified in terms of the model parameters 
(and is independent of the domain size and shape, and of the loading history). 
As the loading increases, the level of damage raises until the maximal stress 
that the material can sustain is attained.
The response of the body upon further loading depends on its size relative 
to a regularization parameter $\ell$ specified in the model, 
which can be interpreted as internal characteristic length (another material property, 
as the elastic limit stress).
\medskip

In short rods, of length $L \sim \ell$, the homogeneous damage solutions
(where the damage variable is constant throughout the body) remain stable 
even for extreme loading conditions.
In contrast, in rods made of a stress-softening material (so that the elastic region in 
strain space shrinks as the damage progresses), the homogeneous damage solutions lose their 
stability, allowing the internal variable to continue its growth in narrow bands 
of width comparable to $\ell$.
\medskip

Gradient damage models have thus proved to be a consistent numerical approximation of
the propagation of a pre-existing crack in the Griffith model for brittle fracture~\cite{griffith1921vi,francfort1998revisiting}, 
to provide a mechanism for crack nucleation in a faultless material without geometrical 
singularities~\cite{chambolle2008nucleation},
and to capture size effects and softening properties which are significant in the 
behaviour of concrete, rock, and biomaterials.
In addition, gradient damage models overcome the spurious mesh dependency observed 
in local damage models~\cite{marigo1989constitutive,benallal1993,pham2010approche},
since the addition of a gradient term on the damage variable leads to a dissipated energy in
the localized solutions which is essentially proportional to the area of the crack,
as opposed to the failure without dissipated energy observed in the local models.
The initiation of cracks in more complicated three-dimensional geometries has been 
studied, e.g., in~\cite{tanne2018crack,kumar2020revisiting}, and a unified treatment 
of cohesive fracture and nucleation with an arbitrary strength surface is given 
in~\cite{bourdin2025}.
It was shown in \cite{bonnetier14faults} that variants of gradient damage models including 
plasticity and visco-elasticity are compatible with descriptions of the formation of 
geological faults.
\medskip

In such gradient damage models,
the capability for crack nucleation is associated to the 
loss of stability of the homogeneous damage state and the emergence of 
localized damage solutions. 
The previously mentioned stability 
analyses~\cite{pham2011stability,pham2013stability,marigo2016overview} 
have been conducted  mainly for uniaxial tension tests, using the variational inequalities 
of the formally derived one-dimensional gradient damage model. 
A validation of the important dimension reduction from 3D to 1D is thus desirable. 
This is the purpose of this current work:
we prove the $\Gamma$-convergence \cite{braides2002gamma,dalMaso2012,ambrosio2000calculus}
of the three-dimensional model to the one-dimensional gradient damage model in the slender 
rod limit.
\medskip

    An application of gradient damage models to bulk degradation of the rock mass 
    in underground mining, with a numerical observation of surface subsidence, has been proposed in~\cite{bonnetier2025gradient}.
    The present work has been motivated by the use that will be made of the simplified one-dimensional gradient damage model in the identification of parameters from uniaxial compression tests. That is a necessary step towards the more systematic study of three-dimensional damage models in the block caving problem, which is being studied by the CMM (Center for Mathematical Modeling, Universidad de Chile) Mining group.

\medskip
The precise formulation of the mathematical problem studied in this article is as follows: consider a sequence of cylindrical domains
\[
\Omega_j := \{\boldsymbol{x}=(x,y,z)\in \mathbb{R}^3: x^2+y^2\leq R_j, \hspace{0.2cm} 0\leq z\leq L\}, 
\quad j\in \N, \quad \text{with} \quad \delta_j := \frac{R_j}{L} \xrightarrow{j\to\infty} 0
\]
where \( L > 0 \) is fixed and \( R_j \to 0 \). 
An axial displacement \(\tilde{t}\) is imposed at the top boundary \(z = L\), 
producing a fixed, uniform, axial strain
\[
\varepsilon_z = \frac{\tilde{t}}{L}, \quad \text{uniform for all } j\in\mathbb{N}.
\]

The total energy associated with each configuration \((\tilde{\boldsymbol{u}}, \tilde{\alpha})\), 
with \( \tilde{\boldsymbol{u}} : \Omega_j \to \mathbb{R}^3 \) the displacement 
and \( \tilde{\alpha} : \Omega_j \to [0,1] \) the damage variable, is given by the functional:
\begin{align}
\label{eq:energy_physical}
\widetilde{E}_j[\tilde{\boldsymbol{u}}, \tilde{\alpha} ] = \int_{\Omega_j} \frac{1}{2} a_\eta(\tilde{\alpha} (\boldsymbol{X})) A \varepsilon : \varepsilon + w(\tilde{\alpha} (\boldsymbol{X})) + \frac{1}{2} w_1 \ell^2 |\nabla \tilde{\alpha} (\boldsymbol{X})|^2 \, d\boldsymbol{X},
\end{align}
where \(A\) denotes the linear elasticity tensor, $\varepsilon$ the strain tensor, 
\(a_\eta\) a degradation function, \(w\) a damage energy density, 
and \(\ell\) the fixed length-scale parameter (see Section \ref{se:preliminaries} for more details.)
We define the set of admissible configurations by
\begin{align} \label{eq:admissible}
\mathcal{A} := \left\{ (\boldsymbol{u}, \alpha) \in H^1\left(\Omega, \mathbb{R}^3\right) \times H^1(\Omega) \mid 0 \leq \alpha \leq 1 \text{ a.e.},\, u_3(\cdot,0) = 0,\, u_3(\cdot,1) = -\varepsilon_z \right\},
\end{align}
where $\Omega$ is the unit cylinder and $\displaystyle{\boldsymbol{u}=
 \big (u_1(\vec x), u_2(\vec x), u_3(\vec x)\big )= 
\left(\frac{\tilde{u}_1}{R_j}, \frac{\tilde{u}_2}{R_j}, \frac{\tilde{u}_3}{L}\right)}$ 
is the non-dimensionalized displacement, defined on the rescaled coordinates 
$\displaystyle{\boldsymbol{x}=\left( \frac{X}{R_j}, \frac{Y}{R_j}, \frac{Z}{L}\right)}$.
Additionally, we introduce the reduced ansatz space: 
\begin{align} \label{eq:uniaxial_space}
\mathcal{S} := \left\{(\vec u, \alpha)\in \mathcal{A} \mid \exists\, \bar{u}, \bar{\alpha} \in H^1(0,1) \text{ such that } u_3(x,y,z) = \bar{u}(z),\ \alpha(x,y,z) = \bar{\alpha}(z) \text{ a.e.} \right\}.
\end{align}
If a pair $(\vec u,\alpha)$ belongs to $\mathcal{S}$ then, with a slight abuse of notation, 
we write $u_3'(z)$ to denote $\displaystyle \frac{d \bar u_3}{dz}(z)$, in view of the definition of $\mathcal{S}$.
More generally, throughout the text, we consider any function $u \in L^2(0,1)$ as a function in $L^2(\Omega)$ and write indifferently $u(z) = u(\vec{x}) = u(x,y,z)$.
\medskip

Let $E_j$ denote the energy per unit volume $\displaystyle{\widetilde{E}_j\setminus |\Omega_j|}$
(the expression of which in the new variables is given in \eqref{eq:EjInExtense}).
We prove two main results~:
Firstly, we show that the sequence of functionals \((E_j)\) $\Gamma$-converges 
to the effective one-dimensional functional \(E_\infty\) defined by
\begin{align}
    \label{eq:defEinfty}
E_\infty [\vec u, \alpha] =
\begin{cases}
\displaystyle \int_0^1 a_\eta(\alpha(z)) \frac{1}{2} E  |u_3'(z)|^2 + w(\alpha(z)) + \frac{1}{2} w_1 \left( \frac{\ell}{L} \right)^2 |\alpha'(z)|^2 \, dz,
& ( \vec u,\alpha)\in \mathcal{S},
\\
+\infty, & (\vec u,\alpha) \in \mathcal{A}\setminus \mathcal{S}.
\end{cases}
\end{align}
More precisely, we prove that

\begin{theorem}\label{Teo2}
\begin{enumerate}
\item[i)] {($\Gamma$-$\liminf$ inequality)}

Let \( \{( \boldsymbol{u}^{(j)}, \alpha^{(j)})\}_{j \in \mathbb{N}} \) be a sequence 
in~\( \mathcal{A} \). 
Assume that for some $\left( \widehat{\vec u}, \widehat{\alpha} \right) \in \mathcal{A}$,
the sequence $\Big ( \overline{u}_3^{(j)} \Big )_{j\in\N}$ defined by \eqref{eq:horizontal_averages_u3}
converges weakly in $L^2(\Omega)$ to $\widehat{u}_3$, and that $\alpha^{(j)}$ converges strongly in $L^2(\Omega)$ to $\widehat{\alpha}$.
Then 
\begin{align*}
	 E_\infty[\widehat {\vec u}, \widehat \alpha]
	 \leq \liminf_{j\to\infty} E_j[\vec u^{(j)}, \alpha^{(j)}].
\end{align*}

\item[ii)] {(\, $\Gamma$-$\limsup$ inequality)}
 
Given any $(\vec u,\alpha)\in \mathcal{S}$, it is possible to construct a sequence $(\boldsymbol{u}^{(j)},\alpha^{(j)})$ in $\mathcal{A}$ such that
 $$ \big (u_3^{(j)}, \alpha^{(j)}\big) \quad \text{converges strongly in } L^{2}(\Omega)\times L^2(\Omega)
 \text{ to } \big (u_3, \alpha\big ), $$
    and
    \[
    \lim_{j \to \infty} E_j[{\vec u}^{(j)}, \alpha^{(j)}] =\int_0^1 a_\eta\big(\alpha(z)\big) \cdot  \frac{1}{2} E u_3'(z)^2 + w\big(\alpha(z)\big) + \frac{w_1 \ell^2}{2L^2} |\alpha'(z)|^2  \, \mathrm{d}z.
    \]
\end{enumerate}
\end{theorem}
\medskip 

Secondly, we prove that if~$( \boldsymbol{u}^{(j)}, \alpha^{(j)})_{j \in \mathbb{N}}$
is a sequence of minimizers of~$E_j$, a stronger result holds
\begin{theorem}\label{Main}
Suppose that for each $j\in\N$
\begin{eqnarray*}
(\vec u^{(j)}, \alpha^{(j)}) \ \text{minimizes}\  E_j[\vec u, \alpha]
\ \text{in}\ \mathcal A.
\end{eqnarray*}
Then there exists a pair $(\widehat{\vec u},\widehat{\alpha})$ in $\mathcal {S}$
such that (for a subsequence)
\begin{eqnarray}
            &&
                \nonumber
            \int_\Omega |u_3^{(j)}(x,y,z) - \widehat{u}_3(z)|^2 d\vec x \longrightarrow 0, \quad
          \int_\Omega \left |\frac{\partial u_3^{(j)}}{\partial z}(x,y,z) - \frac{d \widehat{u}_3}{d z}(z)\right |^2d\vec x
          \longrightarrow{0},
          \\
          &&
            \label{eq:strainL2converges}
          \int_\Omega \left |\frac{\partial u_1^{(j)}}{\partial x}(x,y,z) + \nu \frac{d \widehat{u}_3}{d z}(z)\right |^2d\vec x
          \longrightarrow{0},
          \quad
          \int_\Omega \left |\frac{\partial u_2^{(j)}}{\partial y}(x,y,z) + \nu \frac{d \widehat{u}_3}{d z}(z)\right |^2d\vec x
          \longrightarrow{0},
          \\
          &&
           \nonumber
          \int_\Omega \left | \frac{\partial u_1^{(j)}}{\partial y} + \frac{\partial u_2^{(j)}}{\partial x}\right|^2
          +\left |\delta_j \frac{\partial u_1^{(j)}}{\partial z} + \delta_j^{-1} \frac{\partial u_3^{(j)}}{\partial x} \right|^2
          +\left |\delta_j \frac{\partial u_2^{(j)}}{\partial z} + \delta_j^{-1} \frac{\partial u_3^{(j)}}{\partial y} \right|^2d\vec x
          \longrightarrow 0,
    \\
    &&
        \nonumber
    \qquad \alpha^{(j)} \xrightarrow{H^1 (\Omega)} \widehat \alpha,
          \quad \text{and}\quad (\widehat{\vec u}, \widehat{\alpha})\ \text{minimizes}\ E_\infty[\vec u, \alpha]\ \text{in}\ \mathcal {S}.
\end{eqnarray}
\end{theorem}

\begin{remark*}
\begin{itemize}
\item[1.]
 Theorem~\ref{Teo2} is a $\Gamma$-convergence result, but for a non-standard topology. Its proof is based
 on Theorem~3 below, a compactness result, which characterizes this `dimension reduction' topology.
 The latter Theorem indeed describes the structure of weak limits of admissible $(\vec{u}^{(j)}, \alpha^{(j)})$'s
when only some of the components of the displacement fields or their derivatives  can be expected to be controlled, 
when the energies $E(\vec{u}^{(j)}, 
\alpha^{(j)})$ are uniformly bounded.
\item[2.]   
Concerning Theorem~2, its proof combines Theorem~3
with a precise lower bound on the energies, when the convergences 
in Theorem~3 hold. This lower bound is derived in Proposition~1.
\item[3.] In addition, under the hypotheses of Theorem~2,
we show strong convergence of the minimizers 
(or, more precisely, on the components $u_3^{(j)}$, on the strain tensors $\varepsilon^{(j)}$, and on the damage variables $\alpha^{(j)}$ and their gradients),
another particular feature of the underlying `dimension reduction' topology.
\item[4.]
As another feature of this topology,
one can choose $u_3^{(j)}(x,y,z) = \widehat{u}_3(z)$ to construct a recovery sequence
in the proof of Theorem~1, i.e., a function which is independent of $j$ (and of $x,y$). The horizontal components $u_1^{(j)}$, $u_2^{(j)}$ do depend on $j$ in the recovery sequence, even with gradients that grow unbounded as $j\to\infty$, but their contribution to the energy becomes negligible in the slender rod limit because of the prefactor $\delta_j$ in the rescaled shear strains in \eqref{eq:strainL2converges}, \eqref{eq:EjInExtense}. The system is dominated by the behaviour of the axial displacement $u_3$, to which the horizontal components are able to adjust, guided by the energy-minimality criterion of reducing, as much as possible, the shear strains (see Equations \eqref{eq:aspiration_UB}, \eqref{eq:horizontal_adapt}, \eqref{eq:test_function}, and \eqref{eq:uniaxial_strain}).
This is consistent with the way in which the limit functional $E_\infty$  depends on the vectorial displacement field $\vec u$: through the axial component $u_3$ only.
\item[5.] In this work, the internal characteristic length~$\ell$ and
the regularization parameter $\eta$ are fixed positive quantities,
so that for fixed damage $\alpha$, the associated elastic displacement
is the solution to an elliptic PDE. It would be interesting to study
under which regimes the dimension reduction could be performed, when
the parameters $\ell_j$, $\eta_j$ are allowed to tend to 0. 
\end{itemize}
\end{remark*}
\medskip

An extensive literature is available for the rigorous derivation of reduced models 
obtained from three-dimensional elasticity (see, e.g., \cite[Chapter 16]{antman2005} 
for linear theories of rods; \cite[Chapters A.1 and B.5]{ciarlet2021volII} for Kirchhoff-Love and von K\'arm\'an equations in the theory of plates,
\cite{FJM2006} and~\cite{neukammRichter2024} for models for plates and rods in the context 
of nonlinear elasticity).
The techniques we use here are inspired by the rigorous derivation of the model for fracture 
and delamination of a thin plate on an elastic foundation, proposed in~\cite{baldelli2014},
and by the dimension reduction analysis of a brittle Kirchhoff-Love plate in the SBD setting
derived~in \cite{babadjianHenao2016}.
\medskip

The structure of the article is as follows: 
Section 2 presents preliminary definitions, notation, and the assumptions under
which our analysis is carried out.
In Section 3, we detail the non-dimensionalization and rescaling process.
In Section~4, we proceed with a discussion on compactness results, while
Section 5 is dedicated to deriving the lower and upper bounds for the energy 
functionals $E_j$.
Finally, the proof of Theorem~\ref{Main} is assembled in Section~\ref{se:minimizers}.

\section{Preliminaries}
\label{se:preliminaries}

In this section  we introduce the definitions, notation, and assumptions 
on which we base the reduction of the three-dimensional gradient damage model 
to a one-dimensional model.

\subsubsection*{The domain and admissible displacements}

For a slender rod with \( R \ll L \), we define the domain:
\[
\tilde{\Omega} := \{ \boldsymbol{X} = (X, Y, Z) \in \mathbb{R}^3 \mid X^2 + Y^2 < R^2, \ 0 < Z < L \}
\]
(tildes are used for sets and functions in the physical domain, 
and will be dropped after non-dimensionalization; 
vector quantities are represented in boldface).
We denote by $\tilde{u} \in H(\tilde{\Omega}; \mathbb{R}^3)$
\begin{eqnarray*}
        \tilde{u}:&& \tilde{\Omega} \to \mathbb{R}^3, \quad \\
   && \boldsymbol{X} \longmapsto\tilde{u}(\boldsymbol{X}) = (\tilde{u}_1(\boldsymbol{X}), \tilde{u}_2(\boldsymbol{X}), \tilde{u}_3(\boldsymbol{X}   ))
\end{eqnarray*}
the displacement of each material point  relative to its initial position,
while $\tilde{\alpha}\in H^1(\tilde(\Omega))$ denotes the damage.
We assume that the rod $\Omega$ is subject to the following mixed boundary conditions
\begin{eqnarray} \label{bc's}
\tilde{u}_3 (X, Y, L) = -\tilde{t}, \quad \tilde{u}_3 (X, Y, 0) = 0, \quad \forall X, Y \text{ such that } X^2 + Y^2 < R^2.
\end{eqnarray}
Our analysis is motivated by the modeling of uniaxial compressive tests ($\tilde t>0$), 
however, it remains valid for a rod in tension ($\tilde t<0$).
\medskip

The equilibrium state under the action of $\tilde{t}$ is defined as the
minimizer of the following 3D gradient damage energy, among all fields 
$(\tilde{\bf u},\alpha) \in H(\tilde{\Omega}; \mathbb{R}^3) \times H^1(\tilde{\Omega})$
that statisfy the boundary conditions~\eqref{bc's}.
\[
\int_{\tilde{\Omega}} \frac{1}{2} a_\eta(\tilde \alpha(\boldsymbol{X})) A \varepsilon : \varepsilon + w(\tilde{\alpha}(\boldsymbol{X})) + \frac{1}{2} w_1 \ell^2 |\nabla \tilde{\alpha}(\boldsymbol{X})|^2 \, d\boldsymbol{X}
\]
Each term is explained below.

\subsubsection*{The strain tensor $\varepsilon$}

The strain tensor $\varepsilon$, associated to the deformation of the body, 
is the symmetric rank-2 tensor defined by
\begin{equation*}
    \varepsilon = \frac{1}{2} \left( \nabla \tilde{u} + \nabla \tilde{u}^{T} \right)
    =
    \begin{bmatrix}
        \varepsilon_{11} & \varepsilon_{12} & \varepsilon_{13} \\
        \varepsilon_{21} & \varepsilon_{22} & \varepsilon_{23} \\
        \varepsilon_{31} & \varepsilon_{32} & \varepsilon_{33}
    \end{bmatrix},
    \qquad
    \nabla \tilde{u} =
      \begin{bmatrix}
        \displaystyle  \frac{\partial \tilde{u}_1}{\partial X} & \displaystyle \frac{\partial \tilde{u}_1}{\partial Y} & \displaystyle \frac{\partial \tilde{u}_1}{\partial Z} \\\\
        \displaystyle  \frac{\partial \tilde{u}_2}{\partial X} & \displaystyle  \frac{\partial \tilde{u}_2}{\partial Y} & \displaystyle  \frac{\partial \tilde{u}_2}{\partial Z} \\\\
        \displaystyle  \frac{\partial \tilde{u}_3}{\partial X} & \displaystyle  \frac{\partial \tilde{u}_3}{\partial Y} & \displaystyle \frac{\partial \tilde{u}_3}{\partial Z}
    \end{bmatrix}.
\end{equation*}

\subsubsection*{The elastic constants}

Before damage, the material is governed by an isotropic linear elastic
stress-strain relation, of the form
\begin{eqnarray*}
\sigma &=& A \varepsilon \;:=\; 2\mu \varepsilon + \lambda (\text{tr} \varepsilon), I
\end{eqnarray*}
where the Lam\'e coefficients \( \lambda \) and \( \mu\) respectively measure 
the material's volumetric (compressive) response, and the material's rigidity. 
The Young's modulus \( E \) and Poisson's ratio \( \nu \) are given by
\begin{eqnarray}\label{Young}
    E \;=\; \frac{\mu(3\lambda + 2\mu)}{\lambda + \mu} &\quad\text{and}\quad& 
    \nu \;=\; \frac{\lambda}{2(\lambda + \mu)}.
\end{eqnarray}

\subsubsection*{The damage variable $\tilde{\alpha}$}

The scalar field $\tilde{\alpha} \in H^1(\tilde{\Omega})$
\begin{eqnarray*}
\tilde{\alpha}:&&\tilde{\Omega} \to [0,1],\\
    &&\boldsymbol{X} \longmapsto \tilde{\alpha}(\boldsymbol{X}) = \tilde{\alpha}(X, Y, Z).
\end{eqnarray*}
describes the distribution of material damage in the physical domain $\tilde{\Omega}$.
It can be interpreted as the impact on the macroscopic stiffness of the medium
of the presence of micro-cracks, of loss of internal cohesion, or of localized 
structural degradation. In our context,
\begin{itemize}
    \item $\tilde{\alpha}(\boldsymbol{X}) = 0$ indicates that the material is intact at point $\boldsymbol{X}$.
    \item $0 < \tilde{\alpha}(\boldsymbol{X}) < 1$ represents partial damage, which means that some mechanical properties, such as stiffness, have been reduced but are not completely lost. 
\item $\tilde{\alpha}(\boldsymbol{X}) = 1$ indicates that the material is completely damaged.
\end{itemize}
Note that $\tilde{\alpha}$ can vary in both the transverse ($X, Y$) 
and longitudinal ($Z$) directions.

\subsubsection*{The degradation function $a_\eta (\tilde{\alpha})$}

The degradation function $a_\eta~:\; [0,1] \to \mathbb{R}$ describes how the elastic stiffness
of the medium soften as damage increases. We assume that $a$ is a continuous, strictly decreasing function, and 
\begin{eqnarray*}
a_\eta(0) \;=\; 1,
&\;\textrm{and}\;& a_\eta(1) = \eta,
\end{eqnarray*}
so that the medium stiffness is that of the elastic phase when no damage
is present, while it degrades to $\eta A$  when fully damaged.

\subsubsection*{The energy associated with damage $w(\tilde{\alpha})$}

The local damage energy density, 
\begin{eqnarray*}
\begin{array}{ccl}
w~: [0,1] & \longrightarrow & \mathbb{\R},
\\
\tilde{\alpha} & \longrightarrow & w(\tilde{\alpha}),
\end{array}
\end{eqnarray*}
measures the pointwise contribution to the increase in damage energy 
when $\tilde{\alpha}$ grows.
We assume that $w$ is a smooth, increasing function and that
\begin{eqnarray*}
w(0) &=& 0.
\end{eqnarray*}
The material parameter $w_1$ that appears in the expression of total energy 
is defined to be the value of \( w(\tilde \alpha) \) when \( \tilde \alpha = 1 \)
\begin{eqnarray*}
w_1 &=& w(1).
\end{eqnarray*}

\subsubsection*{The damage distribution $w(\tilde{\alpha})$}

In the spirit of Mumford-Shah energies~\cite{ambrosio1990approximation, ambrosio1992approximation, bourdin2000numerical, bourdin2008variational, braides1998approximation, francfort1998revisiting}, the energy functional $E_j$
penalizes large values of $|\nabla \tilde{\alpha}({\bf X})|$, so that
damage remains uniform in large regions of $\tilde{\Omega}$, and smooth
distributions of damage are favored, in particular when the value of $w_1$
is large.
\medskip

\section{Non-dimensionalization and rescaling }

We rescale the original domain \( \tilde{\Omega} \) to a unit cylinder of the form
\[
\Omega:=\{\boldsymbol{x}=(x,y,z)\in \mathbb{R}^3:\  x^2+y^2<1,\ 0<z<1\},
\qquad 
x = \frac{X}{R}, \quad y = \frac{Y}{R}, \quad z = \frac{Z}{L}.
\]
Additionally, we define the aspect ratio $\delta$ and the rescaled damage and displacement 
\[
\delta:= \frac{R}{L},
\qquad \alpha (x,y,z) = \tilde{\alpha} (X,Y,Z),
\]
\[
u_1 (x,y,z) = \frac{\tilde{u}_1 (X,Y,Z)}{R}, \quad
u_2 (x,y,z) = \frac{\tilde{u}_2 (X,Y,Z)}{R}, \quad 
u_3 (x,y,z) = \frac{\tilde{u}_3 (X,Y,Z)}{L}.
\]
The normal and shear strains thus become
\begin{eqnarray*}
\varepsilon_{11} &=& \frac{\partial}{\partial x} \tilde{u}_1 (x,y,z) = R \frac{\partial u_1}{\partial x}\frac{\partial x}{\partial X} =  \frac{\partial u_1}{\partial x},
\quad \varepsilon_{22} = \frac{\partial u_2}{\partial y}, 
\quad \varepsilon_{33} = \frac{\partial u_3}{\partial z},
\\[3pt]
\varepsilon_{12} &=& \varepsilon_{21} =  \frac{1}{2} \left( \frac{\partial \tilde{u}_1}{\partial Y} + \frac{\partial \tilde{u}_2}{\partial X} \right) = \frac{1}{2} \left( R \frac{\partial u_1}{\partial Y}\frac{\partial Y}{\partial y} + R \frac{\partial u_2}{\partial X}\frac{\partial X}{\partial x} \right) = \frac{1}{2} \left( \frac{\partial u_1}{\partial y} + \frac{\partial u_2}{\partial x} \right),
\\[3pt]
\varepsilon_{13}  &=& \varepsilon_{31} 
= \frac{1}{2} \left( \frac{\partial \tilde{u}_1}{\partial Z} 
+ \frac{\partial \tilde{u}_3}{\partial X} \right) 
= \frac{1}{2} \left( R \frac{\partial u_1}{\partial z}\frac{\partial z}{\partial Z} 
+ L \frac{\partial u_3}{\partial x}\frac{\partial x}{\partial X} \right) 
= \frac{1}{2} \left( \frac{R}{L} \frac{\partial u_1}{\partial z} 
+ \frac{L}{R} \frac{\partial u_3}{\partial x} \right)
\\[3pt]
&=& \frac{1}{2} \left( \delta \frac{\partial u_1}{\partial z} + \delta^{-1} \frac{\partial u_3}{\partial x} \right),
\qquad 
\varepsilon_{23} = \varepsilon_{32} = \frac{1}{2} \left( \delta\frac{\partial u_2}{\partial z} + \delta^{-1}\frac{\partial u_3}{\partial y} \right).
\end{eqnarray*}

And we obtain the following expressions for the elastic energy density~:
\begin{eqnarray*}
\mu \varepsilon : \varepsilon &=& \mu \left( \varepsilon_{11}^2 + \varepsilon_{12}^2 + \varepsilon_{13}^2 + \varepsilon_{22}^2 + \varepsilon_{23}^2 + \varepsilon_{31}^2 + \varepsilon_{32}^2 + \varepsilon_{33}^2 \right)
\\
&=& \mu \left( \frac{\partial u_1}{\partial x} \right)^2 + \mu \left( \frac{\partial u_2}{\partial y} \right)^2 + \mu \left( \frac{\partial u_3}{\partial z} \right)^2 
\\ 
&&
+ \frac{\mu}{2} \left( \frac{\partial u_1}{\partial y} + \frac{\partial u_2}{\partial x} \right)^2
+ \frac{\mu}{2} \left( \delta \frac{\partial u_1}{\partial z} 
+ \delta^{-1} \frac{\partial u_3}{\partial x} \right)^2
+ \frac{\mu}{2}\left( \delta \frac{\partial u_2}{\partial z} + \delta^{-1} \frac{\partial u_3}{\partial y} \right)^2
\\
\frac{1}{2} A \varepsilon : \varepsilon &=& \mu \varepsilon : \varepsilon + \frac{\lambda}{2} (\text{tr} \varepsilon) I : \varepsilon 
\\ 
&=& \mu \left( \frac{\partial u_1}{\partial x} \right)^2 + \mu \left( \frac{\partial u_2}{\partial y} \right)^2 + \mu \left( \frac{\partial u_3}{\partial z} \right)^2
+ \frac{\lambda}{2} \left( \frac{\partial u_1}{\partial x} + \frac{\partial u_2}{\partial y} + \frac{\partial u_3}{\partial z} \right)^2
\\
&& 
+ \frac{\mu}{2} \left( \frac{\partial u_1}{\partial y} + \frac{\partial u_2}{\partial x} \right)^2 + \frac{\mu}{2} \left( \delta \frac{\partial u_1}{\partial z} + \delta^{-1} \frac{\partial u_3}{\partial x} \right)^2
+ \frac{\mu}{2} \left( \delta \frac{\partial u_2}{\partial z} + \delta^{-1} \frac{\partial u_3}{\partial y} \right)^2.
\end{eqnarray*}

On the other hand, rewriting $\alpha(\boldsymbol{x}) = \tilde{\alpha}({X})$ yields
\begin{eqnarray*}
\frac{\partial \alpha}{\partial x} 
&=& 
\frac{\partial \tilde{\alpha}}{\partial {X}} \cdot \frac{\partial {X}}{\partial x} = \frac{\partial \tilde{\alpha}}{\partial {X}} \cdot R
\quad \;\textrm{and}\;\quad 
\frac{\partial \tilde{\alpha}}{\partial {X}} = \frac{1}{R} \frac{\partial \alpha}{\partial x},
\qquad 
    \frac{\partial \tilde{\alpha}}{\partial {Y}} = \frac{1}{R} \frac{\partial \alpha}{\partial y}, \qquad
\frac{\partial \tilde{\alpha}}{\partial {Z}} = \frac{1}{L} \frac{\partial \alpha}{\partial z},
\\[3pt]
|\nabla \tilde{\alpha} (\boldsymbol{X})|^2 &=&  \frac{1}{R^2} \left( \frac{\partial \alpha}{\partial x} \right)^2 + \frac{1}{R^2} \left( \frac{\partial \alpha}{\partial y} \right)^2 + \frac{1}{L^2} \left( \frac{\partial \alpha}{\partial z} \right)^2
   = \frac{1}{L^2} \left[ \delta^{-2} \left( \frac{\partial \alpha}{\partial x} \right)^2 + \delta^{-2} \left( \frac{\partial \alpha}{\partial y} \right)^2 + \left( \frac{\partial \alpha}{\partial z} \right)^2 \right].
\end{eqnarray*}

As mentioned in the introduction, we consider a  sequence of cylindrical 
domains~$(\Omega_j)_{j \in \mathbb{N}}$
with diameter $\delta_j=R_j/L\xrightarrow{j\to\infty}0$.
In the non-dimensionalized variables, the total energy of damage gradient per unit volume~\eqref{eq:energy_physical} can be rewritten as:
\begin{eqnarray}\label{la1}
E_{j} [\boldsymbol{u}, \alpha] 
&:=&
\frac{1}{\pi}\int_{\Omega} a_{\eta} (\alpha (\boldsymbol{x})) 
\left[
\mu \left( \frac{\partial u_1}{\partial x} \right)^2
+ \mu \left( \frac{\partial u_2}{\partial y} \right)^2
+ \mu \left( \frac{\partial u_3}{\partial z} \right)^2
+ \frac{\lambda}{2} \left( \frac{\partial u_1}{\partial x} + \frac{\partial u_2}{\partial y} + \frac{\partial u_3}{\partial z} \right)^2
\right.\nonumber
\\ 
&&\left.+ \frac{\mu}{2}\left( \frac{\partial u_1}{\partial y}+ \frac{\partial u_2}{\partial x} \right)^2
+ \frac{\mu}{2} \left( \delta_j \frac{\partial u_1}{\partial z} + \delta_j^{-1} \frac{\partial u_3}{\partial x} \right)^2 + \frac{\mu}{2} \left( \delta_j \frac{\partial u_2}{\partial z} + \delta_j^{-1} \frac{\partial u_3}{\partial y} \right)^2\right] 
\nonumber 
\\
\label{eq:EjInExtense}
&&+ w (\alpha (\boldsymbol{x})) + \frac{1}{2} w_1 \left( \frac{l}{L} \right)^2 \left[ \delta_j^{-2} \left( \frac{\partial \alpha}{\partial x} \right)^2 + \delta_j^{-2} \left( \frac{\partial \alpha}{\partial y} \right)^2 + \left( \frac{\partial \alpha}{\partial z} \right)^2 \right] d\boldsymbol{x}.
\end{eqnarray}
\medskip

\section{Compactness}
    \label{se:compactness}

In this section, we characterize the structure of the weak convergence
limits of sequences of admissible fields $({\bf u}^j, \alpha^j)$.

\begin{theorem}\label{Teo1}
Suppose that \( \{( \boldsymbol{u}^{(j)}, \alpha^{(j)})\}_{j \in \mathbb{N}} \) 
is a sequence in \( \mathcal{A} \) such that for some $M >0 $
\[
\forall\; j \in  \mathbb{N},\quad
 E_j [( \boldsymbol{u}^{(j)}, \alpha^{(j)})] \leq M,
\]
with $E_j$ defined by \eqref{eq:EjInExtense}. 
For each \( j \in \mathbb{N} \), let \( \overline{u}_3^{(j)} : (0,1) \to \mathbb{R} \) 
denote the horizontal average of $u_3$, defined by
\begin{align}
    \label{eq:horizontal_averages_u3}
    \overline{u}_3^{(j)} (z) := \fint_{x_1^2+x_2^2 < 1} u_3^{(j)} (x,y,z) \, d\mathcal{H}^2 (x,y).
\end{align}
Then, there exists 
$\widehat{\varepsilon}_{11}, \widehat{\varepsilon}_{22}, \widehat{\varepsilon}_{33} $ in $ L^2(\Omega)$
and $\widehat \alpha \) in \(H^1(\Omega)$, such that, up to extraction of a subsequence (not relabelled), 

\begin{gather*}
\frac{\partial u_1^{(j)}}{\partial x} \rightharpoonup \widehat{\varepsilon}_{11}, \quad
\frac{\partial u_2^{(j)}}{\partial y} \rightharpoonup \widehat{\varepsilon}_{22}, \quad
\frac{\partial u_3^{(j)}}{\partial z} \rightharpoonup \widehat{\varepsilon}_{33},
\\
\alpha^{(j)} \rightharpoonup  \widehat{\alpha} \hspace{0.2cm}\in H^{1}(\Omega),
\quad \alpha^{(j)} \to \widehat\alpha \;\text{a.e.~in}\ \Omega,
\quad 0\leq \widehat{\alpha}\leq 1\ \text{a.e.~in}\ \Omega,
\end{gather*}
and such that $\overline{u}_3^{(j)}$ converges,
weakly in $H^1(0,1)$ and strongly in $L^2(0,1)$,
to the function 
\( \widehat{u}_3 : \Omega \to \mathbb{R} \) defined (up to a Lebesgue-null set) by
\begin{eqnarray}\label{la0}
\widehat{u}_3 (x,y,z) = \widehat{u}_3 (z) := \int_0^z \varepsilon_{33}(s) ds, \quad
\textrm{where}\quad
\varepsilon_{33}(s) := \fint_{x^2 + y^2 < 1} \widehat{\varepsilon}_{33} (x,y,s) 
\, d\mathcal{H}^2 (x,y).
\end{eqnarray}
Furthermore, given any pair of functions $u_1$ and $u_2$ in $H^1(\Omega)$, 
the pair $(\vec{u},\alpha) \in H^1(\Omega; \R^3) \times H^1(\Omega)$ ,
with ${\bf u} = (u_1, u_2, \widehat{u}_3)$, belongs to $\mathcal S$.
\end{theorem}

\begin{proof}
Firstly, we establish the weak convergence of the derivatives of $u^{(j)}$.
Given that
\begin{eqnarray} \label{bound_Ej}
E_j \left[ \left(\boldsymbol{u}^{(j)}, \alpha^{(j)}\right) \right] 
&\leq& 
M \quad \text{for some } M > 0,
\end{eqnarray}
we infer from~\eqref{eq:EjInExtense}that
\[
\frac{1}{\pi} \int_{\Omega} a_{\eta} \left(\alpha^{(j)}(\boldsymbol{x})\right)
\left[
\mu \left( \frac{\partial u_1^{(j)}}{\partial x} \right)^2
+ \mu \left( \frac{\partial u_2^{(j)}}{\partial y} \right)^2
+ \mu \left( \frac{\partial u_3^{(j)}}{\partial z} \right)^2
\right] d\boldsymbol{x}
\leq M.
\]
From the uniform lower bound on $a_\eta$
\[
a_{\eta} \left(\alpha^{(j)}(\boldsymbol{x})\right) 
\geq \eta > 0,
\]
it follows that
\[
\frac{\eta}{\pi}\mu
\left[
\int_{\Omega} \left( \frac{\partial u_1^{(j)}}{\partial x} \right)^2 d\boldsymbol{x}
+ \int_{\Omega} \left( \frac{\partial u_2^{(j)}}{\partial y} \right)^2 d\boldsymbol{x}
+ \int_{\Omega} \left( \frac{\partial u_3^{(j)}}{\partial z} \right)^2 d\boldsymbol{x}
\right] \leq M,
\]
which in turn implies the uniform bounds
\begin{eqnarray*}
\left\| \frac{\partial u_1^{(j)}}{\partial x} \right\|^2_{L^2(\Omega)}
\;+\;\left\| \frac{\partial u_2^{(j)}}{\partial y} \right\|^2_{L^2(\Omega)}
\;+\;\left\| \frac{\partial u_3^{(j)}}{\partial z} \right\|^2_{L^2(\Omega)}
&\leq& 
\frac{M \pi}{\eta \mu}.
\end{eqnarray*}

It follows from the Banach-Alaoglu theorem, that 
there exists functions 
$\widehat{\varepsilon}_{11}, \widehat{\varepsilon}_{22}, \widehat{\varepsilon}_{33} \in L^2(\Omega)$,
such that for a subsequence (not relabelled)
\[
\frac{\partial u_1^{(j)}}{\partial x} \rightharpoonup \widehat{\varepsilon}_{11}, \quad
\frac{\partial u_2^{(j)}}{\partial y} \rightharpoonup \widehat{\varepsilon}_{22}, \quad
\frac{\partial u_3^{(j)}}{\partial z} \rightharpoonup \widehat{\varepsilon}_{33}.
\]

We next show that \( \nabla \alpha^{(j)} \rightharpoonup \nabla \widehat\alpha \), 
and \( \alpha^{(j)} \to \widehat\alpha \) a.e.
Indeed, we see from the uniform bound~\eqref{bound_Ej} that
\begin{align}
    \label{eq:boundsGradAlpha}
\frac{w_1}{2} \left( \frac{l}{L} \right)^2
\left[
\int_{\Omega} \delta_j^{-2} \left( \frac{\partial \alpha^{(j)}}{\partial x} \right)^2
+ \delta_j^{-2} \left( \frac{\partial \alpha^{(j)}}{\partial y} \right)^2
+ \left( \frac{\partial \alpha^{(j)}}{\partial z} \right)^2
d\boldsymbol{x}\right]
\leq M,
\end{align}
and as $\delta_j < 1$, it follows that
\[
\frac{w_1}{2} \left( \frac{l}{L} \right)^2 \|\nabla \alpha^j \|_{L^2(\Omega)}^2 
 \leq M,
\]
so that \( \|\nabla \alpha^{(j)} \|_{L^2(\Omega)} \) is uniformly bounded.
In addition, since $({\bf u}^j, \alpha_j) \in \mathcal{A}$,
$0 \leq \alpha^{(j)} (\boldsymbol{x}) \leq 1$ a.e. in $\Omega$,
so that $\alpha^{(j)}$ is uniformly bounded in $W^{1,2}(\Omega)$.
Invoking again the Banach-Alaoglu Theorem, and the Rellich-Kondrachov Theorem, 
we can extract a subsequence such that  
$\alpha^{(j)} \rightharpoonup \widehat\alpha$ in $W^{1,2}(\Omega)$ and 
$\alpha^{(j)}\to \widehat\alpha$ in $L^2(\Omega)$ and a.e. in $\Omega$.
\medskip

Secondly, we prove that the function $\widehat{u}_3 : (0,1) \to \mathbb{R}$ 
satisfies the boundary conditions~\eqref{bc's} on the vertical displacement, 
and that it is a~$H^1$ function. 
To this end, we note that 
\begin{eqnarray*}
\widehat{u}_3 (0) 
&=& \int_0^0 \varepsilon_{33}(s) ds = 0, \hspace{0.3cm}\text{and,} 
\\
\widehat{u}_3 (1) 
&=& \int_0^1 \varepsilon_{33}(s) ds = \int_0^1 \left[ \fint_{x^2 +y^2 <1} 
\widehat{\varepsilon}_{33} (x,y,s) d\mathcal{H}^2 (x,y) \right] ds
= \frac{1}{\pi}\int_{\Omega}\widehat{\varepsilon}_{33}(\boldsymbol{x})d\boldsymbol{x}
\\
&=& 
\lim_{j\to \infty}\frac{1}{\pi}\int_{\Omega}\frac{\partial u_3^{(j)}}{\partial z}dz
= \frac{1}{\pi}
\lim_{j\to \infty} \fint_{x^2 +y^2 <1}\left(u_3^{(j)}(x,y,1)-u_3^{(j)}(x,y,0)\right)
d\mathcal{H}^2 (x,y)
\\
&=&  -{\varepsilon}_z.
\end{eqnarray*}
\medskip

Thirdly, Jensen's inequality, we infer
\begin{eqnarray*}
\varepsilon_{33}^2 (s) 
&=& 
\left[ \fint_{x^2+y^2<1} \widehat{\varepsilon}_{33} (x,y,s) d\mathcal{H}^2(x,y) \right]^2
\;\leq\; 
\fint_{x^2+y^2<1} \widehat{\varepsilon}_{33}\,^2 (x,y,s) d\mathcal{H}^2(x,y), 
\end{eqnarray*}
so that
\begin{eqnarray*}
\int_0^z \varepsilon_{33}^2 (s) ds 
&\leq& 
\int_0^z \left( \fint_{x^2+y^2<1} \widehat{\varepsilon}_{33}^2 (x,y,s) 
d\mathcal{H}^2(x,y) \right) ds
\leq \fint_\Omega \widehat{\varepsilon}_{33}^2 d\boldsymbol{x}.
\end{eqnarray*}
Using the Cauchy-Schwarz inequality, and considering that $z\in (0,1)$, we see that
\begin{eqnarray*}
\left(\int_{0}^z\varepsilon_{33}(s)ds\right)^2
&\leq& 
z\int_{0}^z \varepsilon_{33}^2(s)ds
\;\leq\; 
\int_{0}^z \varepsilon_{33}^2(s)ds,
\end{eqnarray*}
and thus
\begin{eqnarray*}
\widehat{u}_3(z) &=&
\int_{0}^z\varepsilon_{33}(s)ds 
\;\leq\; \left ( \int_{0}^z \varepsilon_{33}^2(s)ds \right ) ^{1/2}
\;\leq\; \left(\fint_{\Omega} \widehat{\varepsilon}\,^2_{33} ds\right)^{1/2}
\end{eqnarray*}
so that $\widehat{u}_3 \in L^{\infty}(0,1)$ and $\varepsilon_{33}\in L^{2}(0,1)$,
since $\widehat{\varepsilon}_{33} \in L^2(\Omega)$.
In addition,  given $\varphi \in C_c^\infty(0,1)$, we have
\begin{eqnarray*}
\int_0^1 \widehat{u}_3(z) \varphi'(z) \, dz 
&=& 
\int_0^1 \Big( \int_0^z \varepsilon_{33}(s) \, ds \Big) \, \varphi'(z) \, dz
\;=\;
\int_0^1 \int_0^1 \chi_{s<z} \varepsilon_{33}(s) \varphi'(z) \, ds \, dz
\\
&=& \int_0^1 \varepsilon_{33}(s) \Big( \int_s^1 \varphi'(z) \, dz  \Big)\, ds
\;=\; - \int_0^1 \varepsilon_{33}(s) \varphi(s) \, ds.
\end{eqnarray*}
which proves that $\widehat{u}_3 \in H^1(0,1)$  and that ${\widehat{u}_3}' = \varepsilon_{33}$.
\medskip

In order to prove that $\widehat\alpha$
is independent of $x_1$, $x_2$, we note that  
\begin{align} \label{eq:uniaxial_alpha} 
\int_{\Omega} \left(\frac{\partial\widehat\alpha}{\partial x}\right)^2d\boldsymbol{x} 
\leq 
\liminf_{j\to\infty} \int_{\Omega} \left(\frac{\partial\alpha^{(j)}}{\partial x}\right)^2 d\boldsymbol{x} 
\leq 
\liminf_{j\to\infty} \frac{2ML^2\delta_j^2}{w_1\ell^2 }
=0,
\end{align}
by~\eqref{eq:boundsGradAlpha} and the lower semicontinuity of the $L^2$ norm.
Hence $\displaystyle{\frac{\partial\widehat\alpha}{\partial x}=0}$ a.e. Analogously, $\displaystyle{\frac{\partial\widehat\alpha}{\partial y}=0}$ a.e.
\medskip

Next, let $u_1$ and $u_2$ be any pair of functions in $H^1(\Omega)$, 
and define $\vec u\in H^1(\Omega; \R^3) = (u_1, u_2, \widehat{u}_3)$.
From~\eqref{eq:uniaxial_alpha}, it follows that 
$(\vec u, \widehat\alpha) \in \mathcal S$, as claimed.
\medskip

Finally, we prove that the sequence $\overline {u}_3^{(j)}$, 
defined in~\eqref{eq:horizontal_averages_u3}, 
converges weakly to $\widehat{u}_3$ in $L^2(\Omega)$.
To this end, let $\varphi \in L^2({\Omega})$ and consider
\begin{eqnarray*}
J &:=&
\int_{\Omega} \overline{u}_3^{(j)}(x, y, z) \varphi(x, y, z) \, d\mathcal{H}^2(x, y)
\\ 
&=& 
\int_0^1 \int_{x^2+y^2<1} \overline{u}_3^{(j)}(x, y, z) \varphi(x, y, z) \, d\mathcal{H}^2(x, y) 
\\
&=& \int_0^1 \int_{x^2+y^2<1}\left[ \fint_{X^2+Y^2<1 }{u}_3^{(j)}(X,Y,z)
d\mathcal{H}^2(X,Y)\right]\varphi(x, y, z) \, d\mathcal{H}^2(x, y)dz
\\
&=&
\int_0^1 \int_{x^2+y^2<1}\fint_{X^2+Y^2<1 }
\left( \int_{0}^z \partial_3 u_{3}^{(j)}(X,Y,s) \,ds \right)
d\mathcal{H}^2(X,Y) \varphi(x, y, z) \, d\mathcal{H}^2(x, y)dz.
\end{eqnarray*}
Set $\displaystyle{\Phi}(X,Y,s):=\int_{s}^1 \int_{X^2+Y^2<1} \varphi(x,y,z)
d\mathcal{H}(x,y)dz$.
It follows that
\begin{eqnarray*}
J &=&
\int_0^1\fint_{X^2+Y^2<1} \partial_3 u^{(j)}_3 (X,Y,s) \Phi(X,Y,s) \,d\mathcal{H}^2(X,Y) \,ds
\\ 
&\overset{j \to \infty}{\to}&
\int_0^1 \fint_{X^2+Y^2<1 }\widehat{\varepsilon}_{33}(X,Y,s)\Phi(X,Y,s)
d\mathcal{H}^2(X,Y) \,ds
\\
&=&
\int_0^1\int_{x^2+y^2<1}\fint_{X^2+Y^2<1}\int_0^z \widehat{\varepsilon}_{33}(X,Y,s)
d\mathcal{H}^2(X,Y)\,ds \; \varphi(x,y,z)d\mathcal{H}^2(x,y)dz
\\
&=&
\int_0^1\int_{x^2+y^2<1}\int_0^z \fint_{X^2+Y^2<1}\widehat{\varepsilon}_{33}(X,Y,s)
d\mathcal{H}^2(X,Y)\,ds \; \varphi(x,y,z)d\mathcal{H}^2(x,y)dz
\\
&=&
\int_0^1\int_{x^2+y^2<1} \left(\int_{0}^z\varepsilon_{33}(s)\,ds \right) \; 
\varphi(x,y,z)d\mathcal{H}^2(x,y)dz
\\
&=&
\int_{\Omega} \widehat{u}_3(z)\varphi(x,y,z) \,d(x,y,z).
\end{eqnarray*}

On the other hand, we observe that for any $\varphi \in C_c^\infty\big([0,1]\big)$
\begin{eqnarray*}
-\int_0^1 \overline{u}_3^{(j)}(z) \partial_{z}\varphi(z) d\vec x 
&=&
-\int_0^1 \fint_{x^2 +y^2 <1} u_3^{(j)}(x,y,z) \partial_{z} \varphi(z) 
\,d\mathcal H^2(x,y) \,dz 
\\
&=&
\frac{-1}{\pi} \int_\Omega u_3(j)(x,y,z) \partial_{z} \varphi(z) \,d(x,y,z)
\\
&=&  
\frac{1}{\pi}   
\int_\Omega
\partial_{z} u_3^{(j)} (x,y,z) \varphi(z) \,d(x,y,z)
\\ 
&\xrightarrow{j\to\infty}&
\frac{1}{\pi} \int_\Omega \widehat{\varepsilon}_{33}(x,y,z)\varphi (z) d(x,y,z)
\;=\; \int_0^1 \varepsilon_{33}(z) \varphi(z) dz.
\end{eqnarray*}
In other words, for each $j \in \N$, the function $\overline{u}_3^{(j)}$ has a weak derivative, 
and $\partial_{z} \overline{u}_3$ consists of the horizontal averages of $\partial_{z} u_3^{(j)}$.
The above calculation also shows that $\widehat{u}_3 \in H^1(0,1)$ 
with weak derivative $\varepsilon_{33}(z)$.
In addition, the sequence $\big( \overline{u}_3^{(j)}\big)_{j\in \N}$ converges to $\widehat{u}_3$ not only weakly in $L^2(0,1)$ but also weakly in $H^1(0,1)$. Consequently, 
invoking the Rellich-Kondrachov Theorem, a subsequence converges to $\widehat{u}_3$ strongly in $L^2(0,1)$.
\end{proof}


\section{Lower and Upper Energy Bounds} \label{Sec5}
\subsection{Energy lower bound}

In this section, we derive a lower bound on the $E_j$'s in terms of
the limiting energy $E_\infty$. This result is a key step in the proof of 
the lim-inf condition in the $\Gamma$-convergence statement of Theorem~\ref{Teo2}.
It is also helpful in obtaining the weak compactness of a sequence of minimizers
in the proof of Theorem~\ref{Main}.

\begin{proposition}\label{Prop1}
Let \( \{( \boldsymbol{u}^{(j)}, \alpha^{(j)})\}_{j \in \mathbb{N}} \) 
be a sequence in \( \mathcal{A} \). 
Assume that the sequence of slice-averages
$\Big ( \overline{u}_3^{(j)} \Big )_{j\in\N}$ defined by~\eqref{eq:horizontal_averages_u3}
converges weakly in $L^2(\Omega)$ to some $\widehat{u}_3\in H^1(\Omega)$,
which satisfies $\widehat{u}_3(x,y,0)=0$ for a.e.~$x$ and $y$. 
Furthermore, assume that
\begin{gather*}
\frac{\partial u_1^{(j)}}{\partial x} \rightharpoonup \widehat{\varepsilon}_{11}, \quad
\frac{\partial u_2^{(j)}}{\partial y} \rightharpoonup \widehat{\varepsilon}_{22}, \quad
\frac{\partial u_3^{(j)}}{\partial z} \rightharpoonup \widehat{\varepsilon}_{33},
\quad 
\nabla \alpha^{(j)} \rightharpoonup \nabla \widehat \alpha, \quad
\alpha^{(j)} \to \widehat \alpha \text{ pointwise a.e.},
\end{gather*}
for some $\widehat{\eps}_{11}$, $\widehat{\eps}_{22}$, $\widehat{\eps}_{33}$ in $L^2(\Omega)$,
and for some $\widehat \alpha\in H^1(\Omega)$ depending on $z$ only.
Then $\widehat u_3$ is independent of $x$ and $y$, and 
\begin{multline}
\label{eq:conclusion_liminf}
\lf    E_{j} [\boldsymbol{u}^{(j)}, \alpha^{(j)}]
\geq
\int_{0}^1  a_{\eta}(\widehat{\alpha}(z)) \frac{1}{2}E(\widehat{u}_3'(z))^2
+w(\widehat{\alpha}(z))+\frac{1}{2}w_1\left(\frac{l}{L}\right)^2 |\widehat{\alpha}'(z)|^2dz
\\
+
\frac{1}{\pi} \int_{\Omega}a_{\eta}(\widehat{\alpha}(\boldsymbol{x}))
\left[ \frac{\mu}{2} (\widehat{\varepsilon}_{11} - \widehat{\varepsilon}_{22})^2
+ 2(\lambda+\mu)\left(\frac{\widehat{\varepsilon}_{11} + \widehat{\varepsilon}_{22}}{2}+\nu \widehat{\varepsilon}_{33}  \right)^2\right]d\boldsymbol{x}
\\
+\frac{1}{2}\int_{0}^1  a_{\eta}(\widehat{\alpha}(z))E\left(\intp\left( \vareg_{33}(x,y,z)
-\varepsilon_{33}(z) \right)^2d\mathcal{H}^2(x,y)\right)dz
\\
+\lf
\Bigg \{     
\frac{\mu}{\pi} 
\Bigg(
 \int_{\Omega}
\left( \sqrt{a_{\eta} (\widehat{\alpha} (\boldsymbol{x}))} \widehat{\varepsilon}_{11}
- \sqrt{a_{\eta} (\alpha^{(j)} (\boldsymbol{x}))} \frac{\partial u_1^{(j)}}{\partial x} \right)^2d\boldsymbol{x}\\
+\left( \sqrt{a_{\eta} (\widehat{\alpha} (\boldsymbol{x}))} \widehat{\varepsilon}_{22}
- \sqrt{a_{\eta} (\alpha^{(j)} (\boldsymbol{x}))} \frac{\partial u_2^{(j)}}{\partial y} \right)^2d\boldsymbol{x}\\
+
\left( \sqrt{a_{\eta} (\widehat{\alpha} (\boldsymbol{x}))} \widehat{\varepsilon}_{33}
- \sqrt{a_{\eta} (\alpha^{(j)} (\boldsymbol{x}))} \frac{\partial u_3^{(j)}}{\partial z} \right)^2d\boldsymbol{x}
\Bigg)
\\
+\frac{\eta \mu}{2\pi} \int_{\Omega}
\left[ \left( \frac{\partial u_1^{(j)}}{\partial y} + \frac{\partial u_2^{(j)}}{\partial x} \right)^2
+ \left( \delta_j \frac{\partial u_1^{(j)}}{\partial z} 
+ \delta_j^{-1} \frac{\partial u_3^{(j)}}{\partial x} \right)^2 \right.
\left. + \left( \delta_j \frac{\partial u_2^{(j)}}{\partial z} 
+ \delta_j^{-1} \frac{\partial u_3^{(j)}}{\partial y} \right)^2 \right] 
d\boldsymbol{x}
\\
+\frac{1}{2} w_1 \left( \frac{l}{L} \right)^2
\int_{\Omega}
\left[ \delta_j^{-2} \left( \frac{\partial \alpha^{(j)}}{\partial x} \right)^2 + \delta_j^{-2} \left( \frac{\partial \alpha^{(j)}}{\partial y} \right)^2 
+ \left( \frac{\partial \alpha^{(j)}}{\partial z}  
- \frac{\partial {\widehat \alpha}}{\partial z}\right)^2 \right]
d\boldsymbol{x}
\Bigg \}
\end{multline}
where $\varepsilon_{33}(z)$ represents the slice-average of $\widehat \eps_{33}$, 
as defined in~\eqref{la0}.
\end{proposition}


\begin{proof}
Assume that $\{( \boldsymbol{u}^{(j)}, \alpha^{(j)})\}_{j \in \mathbb{N}}$
satisfy the hypotheses of the statement and 
recall that the energy functional $E_j$ is given by
\begin{eqnarray} \label{la2}
 E_j[\vec u^{(j)}, \alpha^{(j)}] = \frac{1}{\pi} \int_{\Omega} a_{\eta}(\alpha^{(j)}(\boldsymbol{x})) \left[ \mu \left(\frac{\partial u_1^{(j)}}{\partial x}\right)^2 + \mu \left(\frac{\partial u_2^{(j)}}{\partial y}\right)^2 + \mu \left(\frac{\partial u_3^{(j)}}{\partial z}\right)^2 \right. \nonumber  \\
\left. + \frac{\lambda}{2} \left(\frac{\partial u_1^{(j)}}{\partial x} + \frac{\partial u_2^{(j)}}{\partial y} + \frac{\partial u_3^{(j)}}{\partial z} \right)^2 + \frac{\mu}{2} \left(\frac{\partial u_1^{(j)}}{\partial y} + \frac{\partial u_2^{(j)}}{\partial x} \right)^2 \right.\nonumber\\
\left. + \frac{\mu}{2} \left(\delta_j \frac{\partial u_1^{(j)}}{\partial z} + \delta_{j}^{-1}\frac{\partial u_3^{(j)}}{\partial x} \right)^2 + \frac{\mu}{2} \left(\delta_j \frac{\partial u_1^{(j)}}{\partial z} + \delta_{j}^{-1} \frac{\partial u_3^{(j)}}{\partial x} \right)^2 \right]\\
+ w(\alpha^{(j)}(\boldsymbol{x})) + \frac{1}{2} w_1 \left(\frac{l}{L}\right)^2 \left[ \delta_{j}^{-1} \left(\frac{\partial \alpha^{(j)}}{\partial x}\right)^2 + \delta_{j}^{-1}\left(\frac{\partial \alpha^{(j)}}{\partial y}\right)^2 + \left(\frac{\partial \alpha^{(j)}}{\partial z}\right)^2 \right] \, d\boldsymbol{x}\nonumber
\end{eqnarray}

In view of the equality
\begin{align}
\int_{\Omega} |f^{(j)}|^2 = \int_{\Omega} | f + (f^{(j)} - f) |^2\nonumber 
= \int_{\Omega} f^2 + 2 \int_{\Omega} (f^{(j)} - f) f
+ \int_{\Omega} |f^{(j)} - f|^2,
\end{align}
we see that if $\displaystyle{f^{(j)}}\rightharpoonup f\in L^2(\Omega)$, then
\begin{eqnarray} \label{la6}
    \int_{\Omega} |f^{(j)}|^2 = \int_\Omega f^2 + \int_{\Omega} (f^{(j)} -f)^2 + o(1)
    \geq \int_\Omega f^2 +o(1).
\end{eqnarray}
where $o(1)$ denotes a quantity that converges to zero as $j\to \infty$.
\medskip

We set 
\begin{eqnarray} \label{def_fx}
f_{x}^{(j)}(\vec{x}) &:=& \sqrt{a_{\eta} (\alpha^{(j)} (\vec{x}))} \, 
\frac{\partial u_1^{(j)}}{\partial x} (\vec{x})
\\
f_x(\vec{x}) &=& \sqrt{a_{\eta} (\widehat{\alpha} (\vec{x}))} \, 
\widehat{\varepsilon}_{11}(\vec{x}).
\end{eqnarray}
Since $a_\eta$ is continuous, $\alpha^{(j)} \to \widehat{\alpha}$ a.e.~in $\Omega$
and $0 \leq \alpha^{(j)} \leq 1$, it follows from the 
Lebesgue dominated convergence Theorem that for any test function $\varphi \in L^2(\Omega)$,
$\sqrt{a_{\eta} (\alpha^{(j)}) } \, \varphi \to 
\sqrt{a_{\eta} (\widehat{\alpha}) }\, \varphi$ strongly in $L^2$,
and consequently
\begin{eqnarray*}
   \lim_{j \to \infty}  \int_\Omega f^{(j)}_x(\vec{x}) \varphi(\vec{x}) \,d\vec{x}
    &=& \int_\Omega f_x(\vec{x}) \varphi(\vec{x}) \,d\vec{x},
\end{eqnarray*}
as $\displaystyle{\frac{\partial u_1^{(j)}}{\partial x} \rightharpoonup  \widehat{\varepsilon}_{11}}$.
A similar argument applies to show that
\begin{eqnarray} \label{def_fy}
f_{y}^{(j)}(\vec{x}) \;:=\; \sqrt{a_{\eta} (\alpha^{(j)} (\vec{x}))} \, 
\frac{\partial u_2^{(j)}}{\partial y}(\vec{x})
&\rightharpoonup&
f_y(\vec{x}) \;:=\; \sqrt{a_{\eta} (\widehat{\alpha} (\vec{x}))} \, 
\widehat{\varepsilon}_{22}(\vec{x})
\\
\label{def_fz}
f_{z}^{(j)}(\vec{x}) \;:=\; \sqrt{a_{\eta} (\alpha^{(j)} (\vec{x}))} \, 
\frac{\partial u_3^{(j)}}{\partial z}(\vec{x})
&\rightharpoonup&
f_z(\vec{x}) \;:=\; \sqrt{a_{\eta} (\widehat{\alpha} (\vec{x}))} \, 
\widehat{\varepsilon}_{33}(\vec{x}),
\end{eqnarray}
weakly in $L^2(\Omega)$.
Applying~\eqref{la6} to the first term of~\eqref{la2} thus  yields  
\begin{eqnarray*}
\frac{\mu}{\pi} \int_{\Omega} a_{\eta} (\alpha^{(j)} (\boldsymbol{x})) 
\left( \frac{\partial u_1^{(j)}}{\partial x} \right)^2 d\boldsymbol{x}
&=& 
\frac{\mu}{\pi} \int_{\Omega} \left( \sqrt{a_{\eta} (\alpha^{(j)} (\boldsymbol{x}))} \frac{\partial u_1^{(j)}}{\partial x} \right)^2 d\boldsymbol{x}
\\
&=& 
\frac{\mu}{\pi} \int_{\Omega} \left( \sqrt{a_{\eta} (\widehat{\alpha} (\boldsymbol{x}))}\widehat \varepsilon_{11} \right)^2 d\boldsymbol{x}
\\
&& + \frac{\mu}{\pi} \int_{\Omega} \left( \sqrt{a_{\eta} (\widehat{\alpha} (\boldsymbol{x}))} \widehat\varepsilon_{11} - \sqrt{a_{\eta} (\alpha^{(j)} (\boldsymbol{x}))} \frac{\partial u_1^{(j)}}{\partial x} \right)^2 d\boldsymbol{x}+o(1),
\end{eqnarray*}
and analog bounds hold for $\displaystyle\frac{\partial u_2^{(j)}}{\partial y}$ and 
$\displaystyle\frac{\partial u_3^{(j)}}{\partial z}$.

Collecting terms we obtain
\begin{eqnarray*}
\lefteqn{
\frac{\mu}{\pi} \int_{\Omega} a_{\eta} (\alpha^{(j)} (\boldsymbol{x}))
\left[ \left( \frac{\partial u_1^{(j)}}{\partial x} \right)^2
+ \left( \frac{\partial u_2^{(j)}}{\partial y} \right)^2
+ \left( \frac{\partial u_3^{(j)}}{\partial z} \right)^2
\right]d\boldsymbol{x} }
\\
&=&
\frac{\mu}{\pi} \int_{\Omega} a_{\eta} (\widehat{\alpha} (\boldsymbol{x}))
\left[ (\widehat\varepsilon_{11})^2 + (\widehat\varepsilon_{22})^2 + (\widehat\varepsilon_{33})^2 \right]d\boldsymbol{x} 
\\
\label{eq:Laplacian}
&&
+ \frac{\mu}{\pi} \Bigg(    \int_{\Omega}
\left( \sqrt{a_{\eta} (\widehat{\alpha} (\boldsymbol{x}))} \widehat\varepsilon_{11}
- \sqrt{a_{\eta} (\alpha^{(j)} (\boldsymbol{x}))} \frac{\partial u_1^{(j)}}{\partial x} \right)^2 d\boldsymbol{x}+
\int_{\Omega}
\left( \sqrt{a_{\eta} (\widehat{\alpha} (\boldsymbol{x}))} \widehat\varepsilon_{22}
- \sqrt{a_{\eta} (\alpha^{(j)} (\boldsymbol{x}))} \frac{\partial u_2^{(j)}}{\partial y} \right)^2 d\boldsymbol{x}
\\&&+
\int_{\Omega}
\left( \sqrt{a_{\eta} (\widehat{\alpha} (\boldsymbol{x}))} \widehat\varepsilon_{33}
- \sqrt{a_{\eta} (\alpha^{(j)} (\boldsymbol{x}))} \frac{\partial u_3^{(j)}}{\partial z} \right)^2 d\boldsymbol{x}
\Bigg)+o(1).
\end{eqnarray*}

As for the factor of $\lambda$ in~\eqref{la2},
we apply again~\eqref{la6}, with 
$\displaystyle{f_{j} (x) = \sqrt{a_{\eta} (\alpha^{(j)} (\boldsymbol{x}))}
\left( \frac{\partial u_1^{(j)}}{\partial x}
+ \frac{\partial u_2^{(j)}}{\partial y}
+ \frac{\partial u_3^{(j)}}{\partial z} \right)}$ 
and~$f(x) =\sqrt{a_{\eta} (\widehat{\alpha} (\boldsymbol{x}))}
\left(\widehat \varepsilon_{11} + \widehat\varepsilon_{22} + \widehat\varepsilon_{33} \right)$
to show that
\begin{eqnarray*}
\label{eq:trace}
\frac{\lambda}{2\pi} \int_{\Omega} a_{\eta} (\alpha^{(j)} (\boldsymbol{x}))
\left( \frac{\partial u_1^{(j)}}{\partial x}
+ \frac{\partial u_2^{(j)}}{\partial y}
+ \frac{\partial u_3^{(j)}}{\partial z} \right)^2d\boldsymbol{x}
&\geq& 
\frac{\lambda}{2\pi} \int_{\Omega} a_{\eta} (\widehat{\alpha} (\boldsymbol{x}))
\left( \widehat\varepsilon_{11} +\widehat \varepsilon_{22} +\widehat \varepsilon_{33} \right)^2d\boldsymbol{x}+o(1).
\end{eqnarray*}

We note that
\begin{eqnarray*}
\lefteqn{
\mu
\big(\widehat{\varepsilon}_{11}\,^2 + \widehat{\varepsilon}_{22}\,^2 + \widehat{\varepsilon}_{33}\,^2\big) 
+ \frac{\lambda}{2}
\big(\widehat{\varepsilon}_{11} + \widehat{\varepsilon}_{22} + \widehat{\varepsilon}_{33}\big)^2 }
\\
&=  &
\frac{\mu}{2} \big(\widehat{\varepsilon}_{11} + \widehat{\varepsilon}_{22}\big)^2
+ 
\frac{\mu}{2} \big(\widehat{\varepsilon}_{11} - \widehat{\varepsilon}_{22}\big)^2
+
\mu \widehat{\varepsilon}_{33}\,^2
+ \frac{\lambda}{2}
\big(\widehat{\varepsilon}_{11} + \widehat{\varepsilon}_{22} + \widehat{\varepsilon}_{33}\big)^2 
\\
&=&
\left( 2(\lambda + \mu) 
\Big( \frac{\widehat{\varepsilon}_{11} + \widehat{\varepsilon}_{22}}{2} 
+ \nu \widehat{\varepsilon}_{33} \Big)^2
+ \frac{E}{2} \widehat{\varepsilon}_{33}\,^2\right) +\frac{\mu}{2}(\widehat{\varepsilon}_{11}-\widehat{\varepsilon}_{22})^2
\end{eqnarray*}
where the Poisson ratio $\nu$ and the Young modulus $E$ are defined in \eqref{Young}.

We thus can estimate the elastic energy density in~\eqref{la2} as follows,
using the fact that $a_\eta > \eta$,
\begin{eqnarray}
\lefteqn{
\frac{1}{\pi} \int_{\Omega} a_{\eta} (\alpha^{(j)} (\boldsymbol{x}))\left[
\left( \mu \left( \frac{\partial u_1^{(j)}}{\partial x} \right)^2
+ \mu \left( \frac{\partial u_2^{(j)}}{\partial y} \right)^2
+ \mu \left( \frac{\partial u_3^{(j)}}{\partial z} \right)^2 \right)
+ \frac{\lambda}{2} \left( \frac{\partial u_1^{(j)}}{\partial x}
+ \frac{\partial u_2^{(j)}}{\partial y}
+ \frac{\partial u_3^{(j)}}{\partial z} \right)^2\right. } \nonumber
\\ 
\nonumber &&
\left.+ \frac{\mu}{2} \left( \frac{\partial u_1^{(j)}}{\partial y}
+ \frac{\partial u_2^{(j)}}{\partial x} \right)^2
+ \frac{\mu}{2} \left( \delta_j \frac{\partial u_1^{(j)}}{\partial z}
+ \delta_j^{-1} \frac{\partial u_3^{(j)}}{\partial x} \right)^2
+ \frac{\mu}{2} \left( \delta_j \frac{\partial u_2^{(j)}}{\partial z}
+ \delta_j^{-1} \frac{\partial u_3^{(j)}}{\partial y} \right)^2\right]d\boldsymbol{x}
\\
\nonumber
&\geq& 
\frac{1}{\pi} \int_{\Omega} a_{\eta} (\widehat{\alpha} (\boldsymbol{x}))
\left[ \mu \widehat{\varepsilon}_{11}^2 
+ \mu \widehat{\varepsilon}_{22}^2 + \mu \widehat{\varepsilon}_{33}^2
+ \frac{\lambda}{2} (\widehat{\varepsilon}_{11} + \widehat{\varepsilon}_{22} + \widehat{\varepsilon}_{33})^2 \right] 
\,d\boldsymbol{x}
\\
\nonumber
&&+ \frac{\mu}{2\pi} \int_{\Omega} \left(a_{\eta}(\alpha^{(j)})\right)  
\left[ \left( \frac{\partial u_1^{(j)}}{\partial y} 
+ \frac{\partial u_2^{(j)}}{\partial x} \right)^2 
+ \left( \delta_j \frac{\partial u_1^{(j)}}{\partial z}
+ \delta_j^{-1} \frac{\partial u_3^{(j)}}{\partial x} \right)^2  
+ \left( \delta_j \frac{\partial u_2^{(j)}}{\partial z} 
+ \delta_j^{-1} \frac{\partial  u_3^{(j)}}{\partial y} \right)^2 \right] 
\,d\boldsymbol{x}
\\
\nonumber &&
+\; J + o(1)
\\
\nonumber &\geq&
\frac{1}{\pi} \int_{\Omega} a_{\eta} (\widehat{\alpha} (\boldsymbol{x}))
\frac{1}{2} E (\widehat{\varepsilon}_{33})^2 d\boldsymbol{x}
+ \frac{\eta}{\pi} \int_{\Omega}
\frac{\mu}{2} (\widehat{\varepsilon}_{11} - \widehat{\varepsilon}_{22})^2
+ 2(\lambda+\mu)\left(\frac{\widehat{\varepsilon}_{11} + \widehat{\varepsilon}_{22}}{2}+\nu \widehat{\varepsilon}_{33}  \right)^2d\boldsymbol{x} 
\\
\nonumber &&  
+\frac{\mu\eta }{2\pi} \int_{\Omega}
 \left( \frac{\partial u_1^{(j)}}{\partial y} + \frac{\partial u_2^{(j)}}{\partial x} \right)^2
+ \left( \delta_j \frac{\partial u_1^{(j)}}{\partial z} + \delta_j^{-1} \frac{\partial u_3^{(j)}}{\partial x} \right)^2
+ \left( \delta_j \frac{\partial u_2^{(j)}}{\partial z} + \delta_j^{-1} \frac{\partial u_3^{(j)}}{\partial y} \right)^2 d\boldsymbol{x}
\\
&&    \label{eq:Laplacian+trace}
+\;J  +  o(1),
\end{eqnarray}
where we have set
\begin{eqnarray*}
J &=& 
\frac{\mu}{\pi} \int_{\Omega}
\left( \sqrt{a_{\eta} (\widehat{\alpha} (\boldsymbol{x}))} \widehat{\varepsilon}_{11}
- \sqrt{a_{\eta} (\alpha^{(j)} (\boldsymbol{x}))} \frac{\partial u_1^{(j)}}{\partial x} \right)^2 
+
\left( \sqrt{a_{\eta} (\widehat{\alpha} (\boldsymbol{x}))} \widehat{\varepsilon}_{22}
- \sqrt{a_{\eta} (\alpha^{(j)} (\boldsymbol{x}))} \frac{\partial u_2^{(j)}}{\partial y} \right)^2 
\\
&&
+
\left( \sqrt{a_{\eta} (\widehat{\alpha} (\boldsymbol{x}))} \widehat{\varepsilon}_{33}
- \sqrt{a_{\eta} (\alpha^{(j)} (\boldsymbol{x}))} \frac{\partial u_3^{(j)}}{\partial z} \right)^2 
\,d\boldsymbol{x}.
\end{eqnarray*}
\medskip

Let us now focus on the first term in~\eqref{eq:Laplacian+trace}
\begin{eqnarray*}
\displaystyle{\frac{1}{2\pi}\int_{\Omega} a_{\eta}(\widehat{\alpha}(\boldsymbol{x}))
E\, \widehat{\varepsilon}_{33}^{\,2} \,d\boldsymbol{x}}.
\end{eqnarray*}

Recalling the definition ~\eqref{la0} of the slice averages $\eps_{33}$, and 
that $\widehat{\alpha}(\boldsymbol{x})=\widehat{\alpha}(z)$, 
we see that
\begin{eqnarray}
\label{eq:e33hatLB}
\lefteqn{
\frac{1}{2\pi}\int_{\Omega} a_{\eta}(\widehat{\alpha}(\boldsymbol{x}))
E \, \widehat{\varepsilon}_{33}^{\,2} \, d\boldsymbol{x}
\;=\;
\frac{1}{2\pi}\int_{0}^1\int_{x^2+y^2<1} a_{\eta}(\widehat{\alpha}(\boldsymbol{x}))
E\, \widehat{\varepsilon}_{33}^{\,2} \,d\mathcal{H}^2(x,y)\,dz }
\\
\nonumber &=&
\frac{1}{2}\int_{0}^1  a_{\eta}(\widehat{\alpha}(z))
E\left(\intp \vareg_{33}^{\,2}(x,y,z) \,d\mathcal{H}^2(x,y)\right)dz
\\
\nonumber &=&  
\frac{1}{2}
\int_{0}^1  a_{\eta} (\widehat{\alpha}(z)) E\left[ \intp \vareg_{33}^{\,2}(x,y,z) 
+ \varepsilon_{33}^2(z) + \varepsilon_{33}^2(z) - 2\varepsilon_{33}^2(z) 
\right. 
\\
\nonumber &&\hspace{50mm}
\left. 
+ 2\big(\varepsilon_{33}(z) - \varepsilon_{33}(z) \big)\vareg_{33}(x,y,z) 
\,d\mathcal{H}^2(x,y)\right] \,dz
\\
\nonumber &=&
\frac{1}{2}
\int_{0}^1  a_{\eta}(\widehat{\alpha}(z))
E\left[ \intp \varepsilon_{33}^{2}(z )
+ 2\varepsilon_{33}(z) (\vareg_{33}(x,y,z)-\varepsilon_{33}(z)) \right.
\\
\nonumber &&\hspace{50mm}
\left.  + \left( \vareg_{33}^{\,2}(x,y,z) - 2\varepsilon_{33}(z)\vareg _{33}(x,y,z)
+\varepsilon_{33}^2(z) \right) \,d\mathcal{H}^2(x,y)\right] \,dz
\\
\nonumber &=&
\frac{1}{2}
\int_{0}^1  a_{\eta}(\widehat{\alpha}(z))E
\left[ \intp \varepsilon_{33}^{2}(z )
+ 2\varepsilon_{33}(z) (\vareg_{33}(x,y,z)-\varepsilon_{33}(z)) 
\right.
\\
\nonumber &&\hspace{50mm} 
\left. + \left( \vareg_{33}(x,y,z)-\varepsilon_{33}(z) \right)^2 
\,d\mathcal{H}^2(x,y)\right] \,dz
\\
\nonumber &=&
\frac{1}{2}
\int_{0}^1  a_{\eta}(\widehat{\alpha}(z)) E\, \varepsilon_{33}^2(z) \,dz
+ \frac{1}{2} \int_{0}^1  a_{\eta}(\widehat{\alpha}(z))E
\left(\intp\left( \vareg_{33}(x,y,z)-\varepsilon_{33}(z) \right)^2
\, d\mathcal{H}^2(x,y)\right) \,dz,
\end{eqnarray}
where, to obtain the last equality, we infered from~\eqref{la0} that
\begin{eqnarray*}
\left(\intp\varepsilon_{33}(z) (\vareg_{33}(x,y,z)-\varepsilon_{33}(z))
\, d\mathcal{H}^2(x,y)\right) &=& 0.
\end{eqnarray*} 
\medskip

At this point, we set out to prove that 
\begin{eqnarray}\label{la7}
\varepsilon_{33}(z) &=& \frac{d \widehat{u}_3}{dz}(z).
\end{eqnarray}
Note that $\widehat{u}_3$ is the weak limit of the functions $\overline{u}_3^{(j)}$,
so it is independent of $x$ and $y$, see~\eqref{eq:horizontal_averages_u3}.
Given $\varphi \in \mathcal{C}_c(0,1)$, we compute
\begin{eqnarray*}
 \left< \frac{d \widehat{u}_3}{dz}(z), \varphi(z) \right>
&=&
- \int_0^1 \widehat{u}_3(z) \frac{d \varphi}{dz}(z) \,dz
\\
&=&
- \lim_{j \to \infty} 
\int_0^1 \overline{u}_3^{(j)}(z) \frac{d \varphi}{dz}(z) \,dz
\\
&=&
- \lim_{j \to \infty} 
\int_0^1 \left( \fint_{x_1^2+x_2^2<1} u_3^{(j)}(x,y,z) 
\frac{d \varphi}{dz}(z) \,\hau(x,y) \right) \,dz
\\
&=&
\frac{1}{\pi} \lim_{j \to \infty} 
\int_\Omega \partial_{z} u_3^{(j)}(x,y,z) \varphi(z) \,d(x,y,z)
\\
&=&
\frac{1}{\pi}\int_\Omega \widehat{\varepsilon}_{33}(x,y,z) \varphi(z) \,d(x,y,z)
\\
&=&
\int_0^1 \left( \fint_{x_1^2+x_2^2<1} \widehat{\varepsilon}_{33}(x,y,z)  
\,\hau(x,y) \right) \varphi(z)\,dz
\\
&=&
\int_0^1 \varepsilon_{33}(z) \varphi(z) \,dz,
\end{eqnarray*}
which proves the claim~\eqref{la7}.
Summarizing, \eqref{eq:e33hatLB} and \eqref{la7} yield
\begin{eqnarray}\label{eq:Jensen}
\frac{1}{2\pi}\int_{\Omega} a_{\eta}(\widehat{\alpha}(\boldsymbol{x}))E(\widehat{\varepsilon}_{33})^2d\boldsymbol{x}
&=&
\frac{1}{2}\int_{0}^1  a_{\eta}(\widehat{\alpha}(z)) E\, \big(\widehat{u}_3'(z) \big)^2 \,dz
\\
&& 
+ \frac{1}{2}\int_{0}^1  a_{\eta}(\widehat{\alpha}(z))
E\left(\intp\left( \vareg_{33}(x,y,z)-\varepsilon_{33}(z) \right)^2
\,d\mathcal{H}^2(x,y)\right) \,dz.
\nonumber
\end{eqnarray}

Invoking again~\eqref{la6}, and recalling that $\widehat \alpha$ is independent of $x$ and $y$,
we also see that
\begin{eqnarray}
    \label{eq:quadratic-dAlpha_dZ}
    \int_\Omega
    \left ( \frac{\partial \alpha^{(j)}}{\partial z} \right)^2
    =
    \pi \int_0^1 |{\widehat \alpha}'(z)|^2 dz
    +
    \int_\Omega
    \left ( \frac{\partial \alpha^{(j)}}{\partial z}  
    - \frac{\partial {\widehat \alpha}}{\partial z}\right)^2 \,d\boldsymbol{x}
    + o(1).
\end{eqnarray}
Combining~\eqref{eq:Laplacian+trace}, \eqref{eq:Jensen}, and \eqref{eq:quadratic-dAlpha_dZ}
yields~\eqref{eq:conclusion_liminf}, and proves the Proposition.
\end{proof}
\bigskip

\subsection{Energy upper bound} \label{se:EUB}


\begin{proof}[Proof of Theorem \ref{Teo2}]

\noindent{Step 1~: the $\Gamma$-$\liminf$ inequality}

Let $(\vec{u}^{(j)}, \alpha^{(j)})_{j \in \mathbb{N}}$ be a sequence in $\mathcal{A}$,
such that for some $( \widehat{\vec u}, \widehat{\alpha} ) \in \mathbb{A}$
\begin{eqnarray*}
\overline{u}_3^{(j)} &\rightharpoonup& \widehat{u}_3 \quad\textrm{weakly in}\; L^2(0,1)
\\
\alpha^{(j)} &\rightarrow& \widehat{\alpha} \quad\textrm{strongly in}\; L^2(0,1).
\end{eqnarray*}
Without loss of generality (taking a subsequence if necessary), we may assume that 
the $\liminf  E_j [\vec u^{(j)}, \alpha^{(j)}] $ is finite, and is actually a limit.
Hence, Theorem \ref{Teo1} can be invoked and we may assume, up to extraction of
a subsequence, that 
\[
\frac{\partial u_1^{(j)}}{\partial x} \rightharpoonup \widehat{\varepsilon}_{11},
\quad
\frac{\partial u_2^{(j)}}{\partial y} \rightharpoonup \widehat{\varepsilon}_{22},
\quad
\frac{\partial u_3^{(j)}}{\partial z} \rightharpoonup \widehat{\varepsilon}_{33},
\]
as stated in Theorem~\ref{Teo1}. The latter Theorem further shows that
\begin{eqnarray*}
\overline{u}_3^{(j)} &\rightarrow& \widehat{u}_3,
\end{eqnarray*}
strongly in $L^2(0,1)$.
In particular, it follows that the limiting state $\big(  \widehat {\vec u}, \widehat \alpha\big )$ 
belongs to the uniaxial class $\mathcal{S}$ defined by~\eqref{eq:uniaxial_space}.
Invoking again Theorem \ref{Teo1}, passing to a subsequence if necessary, 
we find that
\begin{eqnarray*}
\nabla \alpha^{(j)} \rightharpoonup \nabla \widehat \alpha, \quad \textrm{weakly  in}\; L^2(\Omega),
&\quad&
\alpha^{(j)} \to \widehat \alpha \text{ pointwise a.e.}
\end{eqnarray*}
and that
\[
\frac{\partial u_1^{(j)}}{\partial x} \rightharpoonup \widehat{\varepsilon}_{11}, \quad
\frac{\partial u_2^{(j)}}{\partial y} \rightharpoonup \widehat{\varepsilon}_{22}, \quad
\frac{\partial u_3^{(j)}}{\partial z} \rightharpoonup \widehat{\varepsilon}_{33},
\quad \textrm{weakly  in}\; L^2(\Omega),
\]
for some $\widehat{\eps}_{11}$, $\widehat{\eps}_{22}$, $\widehat{\eps}_{33}$ in $L^2(\Omega)$.
The hypotheses of Proposition \ref{Prop1} are thus satisfied,
and~\eqref{eq:conclusion_liminf} establishes the $\Gamma$-$\liminf$ inequality, 
since the remainder terms in this inequality are all nonnegative.
\bigskip


\noindent{Step 2~: the $\Gamma$-$\limsup$ inequality}

Let $\overline{u}_3 \in H^1(0,1)$ satisfy the boundary compression conditions and 
$\overline{\alpha} \in H^1(0,1)$ such that $0 \leq \overline{\alpha}(z) \leq 1$, a.e.\ $z \in (0,1)$. Ideally, one would like to construct three-dimensional strain fields~$\vec{\varepsilon}^{(j)}$ such that
\begin{align} \label{eq:aspiration_UB}
\varepsilon_{11}^{(j)} = \varepsilon_{22}^{(j)} = -\nu \varepsilon_{33}^{(j)}, 
\quad \varepsilon_{12}^{(j)} = \varepsilon_{13}^{(j)} = \varepsilon_{23}^{(j)} = 0,
\end{align}
with $\varepsilon_{33}^{(j)}$ independent from the transverse coordinates $x, y$. 
\medskip

If we defined, for $\vec{x} = (x, y, z) \in \Omega$, 
\begin{eqnarray*} 
\alpha^{(j)}(\vec {x}) = \overline{\alpha}(z), 
&\;\textrm{and}\;& 
u_3^{(j)}(\vec x) = u_3(x,y,z) = \overline{u}_3(z),
\end{eqnarray*}
then, as
\[
\frac{\partial u_3^{(j)}}{\partial x} = \frac{\partial \overline{u}_3(z)}{\partial x} = 0, 
\qquad 
\frac{\partial u_3^{(j)}}{\partial y} = \frac{\partial \overline{u}_3(z)}{\partial y} = 0,
\]
condition~\eqref{eq:aspiration_UB} would imply that
\begin{eqnarray*}
\varepsilon_{13}^{(j)} &=& \frac{1}{2}(
\delta_j \partial_z u_1^{(j)} + \frac{1}{\delta_j} \partial_x u_3^{(j)})
\;=\; \frac{\delta_j}{2} \partial_z u_1^{(j)}  \;=\; 0,
\end{eqnarray*}
and similarly for $\varepsilon_{23}^j$. Consequently, we would have
\begin{eqnarray*}
\partial_z u_1^{(j)} &=& \partial_z u_2^{(j)} \;=\; 0,
\end{eqnarray*}
so that 
\begin{eqnarray*}
\partial_z \varepsilon_{11}^{(j)} &=&
\partial^2_{zx} u_1^{(j)} \;=\; 
- \nu \partial^2_{zz} u_3^{(j)} \;=\;0.
\end{eqnarray*}
However, $\overline{u}_3$ is not, in general, an affine function of $z$,
so that requiring that $\vec{u}^{(j)}$ and the associated strains~$\vec{\varepsilon}^{(j)}$
satisfy~\eqref{eq:aspiration_UB} is too constraining.
\medskip

To relax this condition, we define
\begin{align}
    \label{eq:horizontal_adapt}
u_1^{(j)}(x, y, z) := -\nu x\, \varepsilon_{33}^{(j)}(z),
\quad 
u_2^{(j)}(x, y, z) := -\nu y\, \varepsilon_{33}^{(j)}(z),
\quad 
u_3^{(j)}(x,y,z) := 
\overline{u}_3 (z),
\end{align}
where $\varepsilon_{33}^{(j)}$ is a smoothened version of $\overline{u}_3'(z)$, 
whose construction will be carried out below.
As for the damage variables, the choice
\begin{eqnarray*}
\alpha^{(j)}(\vec{x}) &=& \widehat{\alpha}(z),
\end{eqnarray*}
i.e., independent of $j$ and of $x,y$, will do the trick.  
The longitudinal shear strain then take the form
\begin{eqnarray*}
\varepsilon_{13}^{(j)} &=& 
\frac{1}{2} \left( \delta \frac{\partial u_1^{(j)}}{\partial z} 
+ \delta^{-1}\frac{\partial  u_3^{(j)}}{\partial x} \right)
\\
&=& 
\frac{1}{2} \left( -\delta \nu x \frac{\partial \varepsilon_{33}^{(j)}}{\partial z}
+ \delta^{-1}\frac{\partial  u_3^{(j)}}{\partial x} \right),
\\
\varepsilon_{23}^{(j)} &=& 
\frac{1}{2} \left( \delta \frac{\partial u_2^{(j)}}{\partial z} 
+ \delta^{-1}\frac{\partial  u_3^{(j)}}{\partial y} \right)
\\
&=& 
\frac{1}{2} \left( -\delta \nu y \frac{\partial \varepsilon_{33}^{(j)}}{\partial z}
+ \delta^{-1}\frac{\partial  u_3^{(j)}}{\partial y} \right).
\end{eqnarray*}
Inserting these test fields in the expression of the energy, we obtain
\begin{eqnarray*}
 E_j[{\vec u}^{(j)}, \alpha] &= &\frac{1}{\pi} \int_{\Omega} 
 a_{\eta}(\overline \alpha(z))
\left[
\mu \left( \frac{\partial u_1^{(j)}}{\partial x} \right)^2
+ \mu \left( \frac{\partial u_2^{(j)}}{\partial y} \right)^2
+ \mu \left( \frac{\partial u_3^{(j)}}{\partial z} \right)^2
\right.  \\
&& \left.
+ \frac{\lambda}{2} \left( \frac{\partial u_1^{(j)}}{\partial x}
+ \frac{\partial u_2^{(j)}}{\partial y}
+ \frac{\partial u_3^{(j)}}{\partial z} \right)^2
+ \frac{\mu}{2} \left( \frac{\partial u_1^{(j)}}{\partial y}
+ \frac{\partial u_2^{(j)}}{\partial x} \right)^2
+ \frac{\mu}{2} \left( \delta_j \frac{\partial u_1^{(j)}}{\partial z}
+ \delta_j^{-1} \frac{\partial u_3^{(j)}}{\partial x} \right)^2
\right.
\\
&& \left.
+ \frac{\mu}{2} \left( \delta_j \frac{\partial u_2^{(j)}}{\partial z}
+ \delta_j^{-1}\frac{\partial u_3^{(j)}}{\partial y} \right)^2\right]
+ w(\overline{\alpha}(z)) 
+ \frac{w_1}{2}  \frac{l^2}{L^2} |\overline{\alpha}'(z)|^2 \,d\boldsymbol{x}
\\
&=& 
\frac{1}{\pi} \int_{\Omega} a_{\eta}(\overline{\alpha}(z)) \left[
\mu \nu^2 \left| \varepsilon_{33}^{(j)} \right|^2
+ \mu \nu^2  \left| \varepsilon_{33}^{(j)} \right|^2
+ \mu \left|{\overline{u}}_3' \right|^2
+ \frac{\lambda}{2} \left( {\overline{u}}_3' - 2\nu \varepsilon_{33}^{(j)} \right)^2  \right.
\\
&&
+ \frac{\mu}{2}\, \nu^2\, \delta_j^2 x^2 \left(
\frac{\partial}{\partial z} \varepsilon_{33}^{(j)}
\right)^2
+ \frac{\mu}{2}\, \nu^2\, \delta_j^2 y^2 \left(
\frac{\partial}{\partial z} \varepsilon_{33}^{(j)}
\right)^2\Bigg]
+ w(\overline{\alpha}(z))
+ \frac{w_1}{2}  \frac{l^2}{L^2} \left|\alpha'(z) \right|^2 \,d\boldsymbol{x}
\\
&=&
\frac{1}{\pi} \int_{\Omega} a_{\eta}(\overline{\alpha}(z))
\Bigg[
\left( \mu (2\nu^2 + 1) + \frac{\lambda}{2} (2\nu - 1)^2 \right) \left| \varepsilon_{33}^{(j)} \right|^2 
+ \Big ( \mu + \frac{\lambda}{2}\Big ) \Big ( |\overline{u}_3'|^2 - |\varepsilon_{33}^{(j)}|^2 \Big ) 
\\
&&
+ 2\lambda\nu  \varepsilon_{33}^{(j)} \Big (
 \varepsilon_{33}^{(j)} - \overline{u}_3' \Big ) 
+ (x^2 + y^2)\, \frac{\mu}{2} \nu^2 \delta_j^2 
\left| \frac{\partial}{\partial z} \varepsilon_{33}^{(j)} \right|^2\Bigg] 
+ w(\overline{\alpha}(z)) + w_1 \frac{\ell^2}{2L^2} \left| \alpha'(z) \right|^2
\,d\boldsymbol{x}.
\end{eqnarray*}
\medskip

Noting that
\begin{eqnarray}\label{la8}
\mu(2\nu^2 + 1) + \frac{\lambda}{2} (2\nu - 1)^2
&=&    
\mu \left( 2\left(\frac{\lambda}{2(\lambda + \mu)}\right) ^2 + 1\right) + \frac{\lambda}{2} \left( \frac{2\lambda}{2(\lambda + \mu)} - 1 \right)^2\nonumber
\\
&=&
\frac{\lambda^2 \mu + 2 \lambda^2 \mu + 4 \lambda\mu^2 
+ 2 \mu^3 + \lambda^3 - 2 \lambda^3 - 2\lambda^2\mu + \lambda^3
+ 2\lambda ^2\mu + \lambda \mu^2}{2(\lambda + \mu)^2}
\nonumber
\\
&=&
\frac{\mu \left( \lambda + \mu \right) \left( 3\lambda + 2\mu \right) }
{2(\lambda + \mu)^2} \;=\; \frac{1}{2}E,
\end{eqnarray}
the expression of $E_j[\vec u^{(j)}, \alpha]$ reduces to
\begin{eqnarray} \label{eq_Ej}
 E_j[\vec u^{(j)}, \alpha]
&=&
\frac{1}{\pi} \int_{\Omega} 
a_{\eta}(\overline{\alpha}(z)) \left[ \frac{E}{2} \left| \varepsilon_{33}^{(j)} \right|^2 
+ \Big( \mu + \frac{\lambda}{2} \Big) \Big( |\overline{u}_3'|^2 - |\varepsilon_{33}^{(j)}|^2 \Big) 
+ 2\lambda\nu  \varepsilon_{33}^{(j)} \Big( \varepsilon_{33}^{(j)} - \overline{u}_3' \Big)
\right . 
\nonumber \\ 
&& 
\left . 
+ (x^2 + y^2) \frac{\mu}{2} \nu^2 \delta_j^2 
\left| \frac{\partial}{\partial z} \varepsilon_{33}^{(j)} \right|^2  \right]
+ w(\overline{\alpha}(z)) 
+ w_1 \frac{\ell^2}{2L^2} \left| \overline{\alpha}'(z) \right|^2 \,d\boldsymbol{x} 
\nonumber \\
&=&
\frac{1}{\pi}\int_{\Omega} 
a_{\eta} (\overline{\alpha}(z))\frac{E}{2}(\overline{u}_3'(z))^2
+ w(\overline{\alpha}(z)) + w_1 \frac{\ell^2}{2L^2} \left| \overline{\alpha}'(z) \right|^2 
\,d\boldsymbol{x} 
\nonumber \\      
&&
+ \frac{1}{\pi}\int_{\Omega} a_{\eta}(\overline{\alpha}(z))\Bigg[
\Bigg( \mu + \frac{\lambda}{2} - \frac{E}{2} \Bigg)
\left( |\overline{u}_3'|^2 -|\varepsilon_{33}^{(j)} |^2\right)
+ 2\lambda\nu  \varepsilon_{33}^{(j)} \Big( \varepsilon_{33}^{(j)} - \overline{u}_3' \Big)
\nonumber \\ 
&& 
+ (x^2 + y^2) \frac{\mu}{2} \nu^2 \delta_j^2 
\left| \frac{\partial}{\partial z} \varepsilon_{33}^{(j)} \right|^2 \Bigg]
\,d\boldsymbol{x}.
\end{eqnarray}
\medskip

We now proceed to construct the functions $\varepsilon_{33}^{(j)}$, so that
$\varepsilon_{33}^{(j)} \to \overline{u}_3^\prime$ in $L^2(0,1)$,
and so that 
\begin{eqnarray*}
\frac{\mu}{2} (x^2 + y^2)\, \nu^2  \left\|\delta_j\, \frac{\partial}{\partial z} \varepsilon_{33}^{(j)} \right\|^2
&\rightarrow& 0.
\end{eqnarray*}
The latter condition requires that $\delta_j \partial_z \varepsilon_{33}^{(j)}$ tends
to 0, and hence that $\partial_z \varepsilon_{33}^{(j)}$ should not grow too rapidly.
The specific construction of $\varepsilon_{33}^{(j)}=v_{k_j}$ must then balance the approximation 
of $\overline{u}_3'(z)$  with a strict control of its derivative.
To this end, we consider a mollifier $\rho \in {\mathcal C}^\infty_0(\mathbb{R})$, 
and set
\begin{eqnarray*}
v_k(z)  &:=& \overline{u}_3'*\rho_{1/\sqrt{k}}(z),
\end{eqnarray*}
where $\rho_s(z) = s^{-1} \rho(z/s)$ for $s > 0$, extending the
definition  of $\overline{u}_3^\prime$ by $0$ outside of $(0,1)$, so that
the convolution is well defined. The functions $v_k$ satisfy
\begin{eqnarray*}
v_k &=&
\overline{u}_3'*\rho_{\frac{1}{\sqrt{k}}}
\;\to\;  \overline{u}_3' 
\quad \textrm{strongly in}\; L^2(0,1)\;\textrm{as}\; k\to \infty,
\end{eqnarray*}
and
\begin{eqnarray*}
\frac{d}{dz}v_k &=& 
\frac{d}{dz}\left (\overline{u}_3'*\rho_{\frac{1}{\sqrt{k}}}\right)
\;=\;
\overline{u}_3'* \left( k\rho'\left(\sqrt{k}z\right)  \right).
\end{eqnarray*}
It follows that
\begin{eqnarray*}
\left\|\frac{d}{dz}v_k \right\|_{L^\infty ( (0,1) )}
&\leq& 
\|\overline{u}_3'(z)\|_{L^1} \left\|k\rho'\left(\sqrt{k}z\right)  \right\|_{L^{\infty}}
\;=\;
k\|\overline{u}_3'(z)\|_{L^1}\left\|\rho'\left(\sqrt{k}z\right)  \right\|_{L^{\infty}}
\;\leq\;  k M\|\overline{u}_3'\|_{L^1}.
\end{eqnarray*}
Given that 
$\overline{u}_3'(z)\in L^2\big ( (0,1)\big ) \subset L^{1}\big ( (0,1)\big )$, we see that
for some constant $C >0$, independent of $k$,
\begin{eqnarray*}
\left\|\frac{d}{dz}v_k \right\|_{L^{\infty}((0,1))} &\leq& Ck.
\end{eqnarray*}

We define $\eps_{33}^{(j)}$ by taking the diagonal subsequence
$\eps_{33}^{(j)} = v_{k_j}$, with the choice
${k_j}= \left\lfloor \delta_j^{-1/2} \right\rfloor$,
which garantees that
\begin{eqnarray*}
\eps_{33}^{(j)} = v_{k_j} \to \overline{u}_3' \quad \text{in } L^2(0,1),
&\;\text{and}\;&
\left\| \frac{\partial}{\partial z} \varepsilon_{33}^{(j)} \right\|_{L^\infty} 
= \left\| \frac{\partial}{\partial z} v_{k_j} \right\|_{L^{\infty}} 
\;\leq\; Ck_j\approx C\delta_{j}^{-1/2},
\end{eqnarray*}
so that
\begin{eqnarray*}
\delta_j^2\left\| \frac{\partial}{\partial z} \varepsilon_{33}^{(j)} \right\|_{L^2}^2 
&\leq&
\delta_j^2 \cdot \delta_j^{-1} \;=\; \delta_j
\;\to\; 0 \quad \textrm{as}\; j \to \infty.
\end{eqnarray*}
We conclude from~\eqref{eq_Ej}that
\begin{eqnarray*}
\lim_{j \to \infty} E_j[\vec u^{(j)}, \alpha] 
&=&
\int_0^1 a_\eta\left( \overline{\alpha}(z) \right) \cdot
\frac{1}{2} E \overline{u}_3'(z)^2 + w\left( \overline{\alpha}(z) \right) +
\frac{w_1 \ell^2}{2 L^2} \left| \overline{\alpha}'(z) \right|^2 \, \mathrm{d}z,
\end{eqnarray*}
as claimed.\\
\end{proof}


\section{Convergence of minimizers} \label{se:minimizers}


\begin{proof}[Proof of Theorem \ref{Main}.]

We consider the functions 
$\boldsymbol{u}_{\text{test}}(\boldsymbol{x})
=(u_1(\boldsymbol{x}),u_2(\boldsymbol{x}),u_3(\boldsymbol{x}))$, 
and  $\alpha_{\text{test}}(\boldsymbol{x})$, defined by
\begin{align}
    \label{eq:test_function}
u_1(x,y,z)=\nu \varepsilon_zx, \quad u_2(x,y,z)=\nu \varepsilon_zy, 
\quad u_3(x,y,z)=-\varepsilon_zz, \quad \alpha_{\text{test}}(x,y,z)\equiv 0.
\end{align}
Note that $\boldsymbol{u}_{\text{test}}$ is the minimizer of the energy when no damage is present, and that Equations \eqref{eq:aspiration_UB} are satisfied.
Since $a_\eta(0)=1$ and $w(0)=0$, see~\eqref{la8},
we evaluate the functional $E_j$ at the test pair 
$\big( \vec u_{\text{test}}, \alpha_{\text{test}} \big)$~:
\begin{eqnarray*}
E_j[\boldsymbol{{u}}_{\text{test}},\alpha_{\text{test}}]
&=&
\frac{1}{\pi}\int_{\Omega} a_{\eta}(\alpha_{\text{test}}(\boldsymbol{x}))\left[ \mu (\nu \varepsilon_z)^2+\mu (\nu \varepsilon_z)^2+\mu \varepsilon_z^2 + \frac{\lambda}{2}\left(\nu \varepsilon_z+\nu \varepsilon_z- \varepsilon_z\right)^2 \right]+w(\alpha_{\text{test}}(\boldsymbol{x})) \,dx
\\
&=&
\frac{1}{\pi}\int_{\Omega}\left[ \mu\varepsilon_z^2(2\nu^2+1)
+ \frac{\lambda}{2}\varepsilon_z^2(2\nu-1)^2 \right] \,dx
\;=\;
\frac{E}{2}\varepsilon_z^2.
\end{eqnarray*}
Since by hypothesis, $(\boldsymbol{u}^{(j)},\alpha^{(j)})$ minimizes 
$E_{j}[\boldsymbol{u},\alpha]$, we infer that
\begin{align}
    \label{eq:aPrioriBound}
E_{j}[\boldsymbol{u}^{(j) },\alpha^{(j)}]
\leq 
E_j[\boldsymbol{u}_{\text{test}},\alpha_{\text{test}}]
\;=\; \frac{E}{2}\varepsilon_z^2,
\end{align}
so that $\displaystyle{E_{j}[\boldsymbol{u}^{(j) },\alpha^{(j)}]\leq M}$ 
is uniformly bounded.
It follows from Theorem~\ref{Teo1} that there exists
$\widehat{u}_3,\hspace{0.1cm} \widehat \alpha$ in $H^1(\Omega)$, depending only on $z$,
and
$\widehat{\varepsilon}_{11}, \widehat{\varepsilon}_{22}, \widehat{\varepsilon}_{33}$ in 
$L^2(\Omega; \mathbb{R}^{3})$, which satisfy~\eqref{la0}, such that
\[ \left\{ \begin{array}{c}
\widehat{u}_3(\cdot, 0) = 0, 
\quad \widehat{u}_3(\cdot, 1) = -\varepsilon_z,
\qquad 0\leq \widehat \alpha\leq 1\quad \text{a.e. in}\; (0,1),
\\[3pt]
\frac{\partial u_1^{(j)}}{\partial x} \rightharpoonup \widehat{\varepsilon}_{11}, \qquad
\frac{\partial u_2^{(j)}}{\partial y} \rightharpoonup \widehat{\varepsilon}_{22}, \qquad
\frac{\partial u_3^{(j)}}{\partial z} \rightharpoonup \widehat{\varepsilon}_{33},
\\[3pt]
\nabla \alpha^{(j)} \rightharpoonup \nabla \widehat\alpha, \qquad
\alpha^{(j)} \to \widehat\alpha \text{ a.e. in}\; \Omega,
\end{array} \right.
\]
and such that the sequence $\overline{u}^{(j)}$, defined in \eqref{eq:horizontal_averages_u3},
converges to $\widehat u_3$ weakly in $H^1\big ( (0,1) \big )$ and strongly in $L^2\big ((0,1) \big )$.
We define the vector-valued map $\widehat{\vec u}\in H^1(\Omega;\R^3)$ by
\begin{eqnarray*}
\widehat{\vec u}(x,y,z) &:=& 
\big(0, 0, \widehat u_3(z) \big), \quad (x,y,z) \in \Omega.
\end{eqnarray*}
Applying Proposition \ref{Prop1}, we obtain
\begin{eqnarray}\label{la9}
\lf E_{j} [\boldsymbol{u}^{(j)}, \alpha^{(j)}] 
&\geq&
E_\infty [\widehat {\vec u}, \widehat\alpha] + J + \lf K^{(j)},
\end{eqnarray}
where 
\begin{eqnarray*}
J &=&
\frac{1}{\pi} \int_{\Omega}a_{\eta}(\widehat{\alpha}(\boldsymbol{x}))
\left[ \frac{\mu}{2} (\widehat{\varepsilon}_{11} - \widehat{\varepsilon}_{22})^2
+ 2(\lambda+\mu)\left(\frac{\widehat{\varepsilon}_{11} + \widehat{\varepsilon}_{22}}{2}+\nu \widehat{\varepsilon}_{33}  \right)^2\right] \,d\boldsymbol{x}
\\
&&
+\frac{1}{2}\int_{0}^1  a_{\eta}(\widehat{\alpha}(z))E
\left(\intp\left( \vareg_{33}(x,y,z)-\varepsilon_{33}(z) \right)^2 \,d\mathcal{H}^2(x,y) 
\,\right)dz 
\\
K^{(j)} &=&
\left\{       \frac{\eta \mu}{2\pi} \int_{\Omega}
\left[ \left( \frac{\partial u_1^{(j)}}{\partial y} 
+ \frac{\partial u_2^{(j)}}{\partial x} \right)^2
+ \left( \delta_j \frac{\partial u_1^{(j)}}{\partial z} + \delta_j^{-1} \frac{\partial u_3^{(j)}}{\partial x} \right)^2 + \left( \delta_j \frac{\partial u_2^{(j)}}{\partial z} + \delta_j^{-1} \frac{\partial u_3^{(j)}}{\partial y} \right)^2 \right] \, d\boldsymbol{x} \right.
\\ 
&&
+ \frac{\mu}{\pi}     
\int_{\Omega}
\left[ 
\left( \sqrt{a_{\eta} (\widehat{\alpha} (\boldsymbol{x}))} \widehat{\varepsilon}_{11}
- \sqrt{a_{\eta} (\alpha^{(j)} (\boldsymbol{x}))} \frac{\partial u_1^{(j)}}{\partial x} \right)^2 
+
\left( \sqrt{a_{\eta} (\widehat{\alpha} (\boldsymbol{x}))} \widehat{\varepsilon}_{22}
- \sqrt{a_{\eta} (\alpha^{(j)} (\boldsymbol{x}))} \frac{\partial u_2^{(j)}}{\partial y} \right)^2
\right.
\\
&&\hspace*{10mm}
+
\left.
\left( \sqrt{a_{\eta} (\widehat{\alpha} (\boldsymbol{x}))} \widehat{\varepsilon}_{33}
- \sqrt{a_{\eta} (\alpha^{(j)} (\boldsymbol{x}))} \frac{\partial u_3^{(j)}}{\partial z} \right)^2
\right] 
\,d \vec{x} 
\\
&&
+ \frac{1}{2} w_1 \left( \frac{l}{L} \right)^2
\int_{\Omega}
\left[ \delta_j^{-2} \left( \frac{\partial \alpha^{(j)}}{\partial x} 
- \frac{\partial {\widehat \alpha}}{\partial x} \right)^2 
+ \delta_j^{-2} \left( \frac{\partial \alpha^{(j)}}{\partial y} 
- \frac{\partial {\widehat \alpha}}{\partial y}\right)^2 
+ \left( \frac{\partial \alpha^{(j)}}{\partial z}  
- \frac{\partial {\widehat \alpha}}{\partial z}\right)^2 \right]
\, d\boldsymbol{x}
\Bigg\}.
\end{eqnarray*}
By Theorem \ref{Teo2}-$(ii)$ applied to  $(\widehat{\vec u},\widehat{\alpha})$, there exist  $\left(\widehat{\boldsymbol{u}}^{(j)}_{r},\widehat{\alpha}^{(j)}_{r}\right)_{j\in \mathbb{N}}$ in $\mathcal{A}$ such that
\begin{eqnarray}\label{la10}
\limsup_{j\to \infty} E_{j}\left[ \widehat{\boldsymbol{u}}^{(j)}_{r},\widehat{\alpha}_{r}^{(j)}\right]= E_\infty [\widehat{\vec u},\widehat{\alpha}].
\end{eqnarray}
Since, for each $j$, $\big (\boldsymbol{u}^{(j)}, \alpha^{(j)}\big )$
is a minimizer,
\[
    \liminf_{j\to\infty} E_{j} [\boldsymbol{u}^{(j)}, \alpha^{(j)}]
    \leq
    \limsup_{j\to\infty} E_{j} [\boldsymbol{u}^{(j)}, \alpha^{(j)}]
\leq
\limsup_{j\to\infty} E_{j} [\widehat{\boldsymbol{u}}_r^{(j)}, \widehat{\alpha}^{(j)}_r].
\]
Thus we deduce from~\eqref{la9} and~\eqref{la10} that
\begin{eqnarray*}
E_\infty [\widehat{\vec u}, \widehat\alpha] \;+\; J \;+\; \;\liminf_{j\to \infty} K^{(j)}
&\leq&
E_\infty [\widehat{\vec u}, \widehat\alpha].
\end{eqnarray*}
Passing to a subsequence, we may assume, without loss of generality, that the lim-inf 
in the above inequalities is actually a limit.
Since expressions of $J$ and $K^{(j)}$ are sums of squares, we obtain
\begin{eqnarray}
    \label{eq:limit_Poisson}
\frac{1}{\pi} \int_{\Omega}a_{\eta}(\widehat{\alpha}(\boldsymbol{x}))
\left[ \frac{\mu}{2} (\widehat{\varepsilon}_{11} - \widehat{\varepsilon}_{22})^2
    + 2(\lambda+\mu)\left(\frac{\widehat{\varepsilon}_{11} + \widehat{\varepsilon}_{22}}{2}+\nu \widehat{\varepsilon}_{33}  \right)^2\right] \,d\boldsymbol{x}
&=& 0.
\end{eqnarray}
\begin{eqnarray} \label{eq:limitJensen}
    \int_{0}^1  a_{\eta}(\widehat{\alpha}(z))E\left(\intp\left( \vareg_{33}(x,y,z)-\varepsilon_{33}(z) \right)^2d\mathcal{H}^2(x,y)\right) \,dz &=& 0,
\end{eqnarray}
\begin{eqnarray} \label{eq:limitShearStrains}
\int_{\Omega}
\left( \frac{\partial u_1^{(j)}}{\partial y} + \frac{\partial u_2^{(j)}}{\partial x} \right)^2
+ \left( \delta_j \frac{\partial u_1^{(j)}}{\partial z} + \delta_j^{-1} \frac{\partial u_3^{(j)}}{\partial x} \right)^2 + \left( \delta_j \frac{\partial u_2^{(j)}}{\partial z} + \delta_j^{-1} \frac{\partial u_3^{(j)}}{\partial y} \right)^2  \,d\boldsymbol{x}
&\to& 0,
\end{eqnarray}
\begin{eqnarray} \label{eq:limitNormalStrains1}
\int_{\Omega}
\left( \sqrt{a_{\eta} (\widehat{\alpha} (\boldsymbol{x}))} \widehat{\varepsilon}_{11}
- \sqrt{a_{\eta} (\alpha^{(j)} (\boldsymbol{x}))} \frac{\partial u_1^{(j)}}{\partial x} \right)^2 \,d\boldsymbol{x}
&\to& 0,
\end{eqnarray}

\begin{eqnarray} \label{eq:limitNormalStrains2}
\int_{\Omega}
\left( \sqrt{a_{\eta} (\widehat{\alpha} (\boldsymbol{x}))} \widehat{\varepsilon}_{22}
- \sqrt{a_{\eta} (\alpha^{(j)} (\boldsymbol{x}))} \frac{\partial u_2^{(j)}}{\partial y} \right)^2 \,d\boldsymbol{x}
&\to& 0,
\end{eqnarray}

\begin{eqnarray} \label{eq:limitNormalStrains3}
\int_{\Omega}
\left( \sqrt{a_{\eta} (\widehat{\alpha} (\boldsymbol{x}))} \widehat{\varepsilon}_{33}
- \sqrt{a_{\eta} (\alpha^{(j)} (\boldsymbol{x}))} \frac{\partial u_3^{(j)}}{\partial z} \right)^2 \,d\boldsymbol{x}
&\to& 0,
\end{eqnarray}

\begin{eqnarray}\label{eq:limitAlpha}
\int_{\Omega}
\delta_j^{-2} \left( \frac{\partial \alpha^{(j)}}{\partial x} -\frac{\partial {\widehat \alpha}}{\partial x} \right)^2 + \delta_j^{-2} \left( \frac{\partial \alpha^{(j)}}{\partial y} 
- \frac{\partial {\widehat \alpha}}{\partial y}\right)^2 
+ \left( \frac{\partial \alpha^{(j)}}{\partial z}  - \frac{\partial {\widehat \alpha}}
{\partial z}\right)^2 \,d\boldsymbol{x}
&\to& 0.
\end{eqnarray}
Using that $\displaystyle{a_{\eta}(\widehat{\alpha}(\boldsymbol{x})) \geq \eta > 0}$,
equation~\eqref{eq:limit_Poisson} yields
\begin{eqnarray*}
\frac{\mu}{2} (\widehat{\varepsilon}_{11} - \widehat{\varepsilon}_{22})^2
    + 2(\lambda+\mu)\left(\frac{\widehat{\varepsilon}_{11} + \widehat{\varepsilon}_{22}}{2}+\nu \widehat{\varepsilon}_{33}  \right)^2=0,
\end{eqnarray*}
and thus $\widehat{\varepsilon}_{11} = \widehat{\varepsilon}_{22} = -\nu \widehat{\varepsilon}_{33}$.
Combining~\eqref{eq:limitJensen}, \eqref{la0}, and~\eqref{la7}
now gives $\vareg_{33}(x,y,z) = \varepsilon_{33}(z) = \widehat{u}'(z)$
and hence
\begin{align}
    \label{eq:uniaxial_strain}
 \widehat{\varepsilon}_{11}=\widehat{\varepsilon}_{22}= -\nu \widehat{u}'(z).
\end{align}
\medskip

On the other hand, by~\eqref{eq:limitNormalStrains1} and Lemma~\ref{Lemma1}
from the Appendix, applied to
\begin{align*}
g^{(j)}(x)=\frac{1}{\sqrt{a_{\eta}\left(\alpha ^{(j)}(\boldsymbol{x})\right)}},
\end{align*}
and to the function $f_x^{(j)}$ defined in \eqref{def_fx}, we see that
\begin{eqnarray*}
    \frac{\partial u_1^{(j)}}{\partial x}=\left(\sqrt{a_{\eta} \left(\alpha ^{(j)}(\boldsymbol{x})\right)} \frac{\partial u_1^{(j)}}{\partial x}\right)\frac{1}{\sqrt{a_{\eta}\left(\alpha ^{(j)}(\boldsymbol{x})\right)}}\xrightarrow{j\to\infty} \vareg_{11},\hspace{0.3cm}\mbox{in } L^{2}(\Omega).
\end{eqnarray*}
In view of the definitions~(\ref{def_fy}-\ref{def_fz})
and of the convergences~(\ref{eq:limitNormalStrains2}-\ref{eq:limitNormalStrains3}),
the same argument applies to show that 
\begin{eqnarray*}
\frac{\partial u_2^{(j)}}{\partial y} \to \widehat{\varepsilon}_{22}
&\textrm{and}&
\frac{\partial u_3^{(j)}}{\partial z} \to \widehat{\varepsilon}_{33},
\end{eqnarray*}
strongly in $L^2$,
and consequently\begin{eqnarray*}
\int_\Omega \left |\frac{\partial u_1^{(j)}}{\partial x}(x,y,z) + \nu \frac{\dd \widehat{u}_3}{\dd z}(z)\right |^2 \dd\vec x
\longrightarrow{0},
\quad
\int_\Omega \left |\frac{\partial u_2^{(j)}}{\partial y}(x,y,z) + \nu \frac{\dd \widehat{u}_3}{\dd z}(z)\right |^2 \dd\vec x
\longrightarrow{0},
\end{eqnarray*}
\begin{eqnarray*}
\int_{\Omega}\left|
\frac{ \partial u_3^{(j)}}{\partial z} (x,y,z) -\frac{\dd \widehat{u}_3}{\dd z}(z)\right|^{2}d(x,y,z)\longrightarrow  0.
\end{eqnarray*}
We now prove that $||u_3^{(j)} - \widehat{u}_3||_{L^2(\Omega)} \to 0$.
Indeed, since
\begin{eqnarray*}
    u_3^{(j)}(x,y,z)-\widehat{u}_3(z)=\int_{s=0}^z \left(\partial_z u_3^{(j)}(x,y,s) -\widehat{u}_3'(s)\right) \,ds,
\end{eqnarray*}
it follows that
\begin{eqnarray*}
||u_3^{(j)} - \widehat{u}_3||_{L^2(\Omega)}^2
&\leq&
\int_{x^2+y^2\leq 1}\int_{z=0}^1 z \int_{s=0}^z\left| \partial_z u_3^{(j)}(x,y,s) 
-\widehat{u}_3'(s)\right|^2 ds\, dz\, \hau (x,y)
\\
&\leq& 
\int_{s=0}^1 \int_{x^2+y^2\leq 1} \left| \partial_z u_3^{(j)}(x,y,s) 
- \widehat{u}_3'(s)\right|^2  \,\hau (x,y)\,ds
\\
&=&
\int_{\Omega} \left| \partial_z u_3^{(j)} - \widehat{u}_3'\right|^2 d\vec x 
\;\;\to\; 0, \quad\textrm{as}\; j \to \infty.
\end{eqnarray*}
\medskip

Finally, the proof that $\big ( \widehat{\vec u}, \widehat{\alpha}\big )$ 
minimizes $E_\infty$ in $\mathcal A$ is standard. We repeat it here for the convenience of the reader, since we work with the non-standard topology of the $L^2$~convergence of the horizontal averages 
$\overline{u}_3^{(j)}$ of the vertical displacements.
Let $\big ( {\vec u}_c, \alpha_c \big )$ be any competitor pair in $\mathcal A$. By the $\Gamma$-$\limsup$ property, there exists a sequence $\Big ( \big ( \vec u_{c,r}^{(j)}, \alpha^{(j)}_{c,r} \big ) \Big )_{j\in \N}$ in $\mathcal{A}$ such that
$$
\limsup_{j\to\infty} E_j\big [ \vec u_{c,r}^{(j)}, \alpha_{c,r}^{(j)}\big ] \leq E_\infty \big [ {\vec u}_c, \alpha_c].
$$
By the $\Gamma$-$\liminf$ property,
$$  E_\infty \big [ \widehat{\vec u}, \widehat{\alpha}\big ]
    \leq \liminf_{j\to\infty} E_j\big [ \vec u^{(j)}, \alpha^{(j)} \big ].
$$
Since, for each $j\in \N$,  $\big (\vec u^{(j)}, \alpha^{(j)}\big) $ minimizes $E_j$ in $\mathcal{A}$, then $E_j\big [ \vec u^{(j)}, \alpha^{(j)} \big ] \leq E_j\big [ \vec u_{c,r}^{(j)}, \alpha_{c,r}^{(j)}\big ]$. This completes the proof.
\end{proof}


\section*{Acknowledgements}

We thank J.-F.~Babadjian for his suggestion of regarding $a_\eta(\partial_{x_i}u_i^{(j)})^2$
as the square of $\sqrt{a_\eta}\partial_{x_i}u_i^{(j)}$ in the proof of the energy lower bound \eqref{eq:Laplacian}, as in \cite{ambrosio1990approximation,ambrosio1992approximation}, which improved a first draft where instead of $|\nabla \alpha|^2$ the choice had been made of an exponent larger than the space dimension, for uniform convergence. 
Also, we are grateful to L.~Scardia for pointing out that our analysis led to a complete proof of $\Gamma$-convergence (Theorem \ref{Teo2}) and not only the result for minimizers (Theorem \ref{Main}), and to B.~Bourdin for insightful discussions on gradient damage models. 

\smallskip
The problem of proving convergence to the uniaxial damage model was motivated by the use that will be made of the one-dimensional equations in the identification of parameters from uniaxial compression tests, as a part of a larger study in underground mining by CMM. The heuristics and perspectives gained from the discussions with S.~Gaete, A.~Jofré, R.~Lecaros, J.~Ortega, J.~Ramírez-Ganga, J.~San Martin, and I.~Vidal are gratefully acknowledged.

\smallskip
E.B.'s work was partially supported by the project ANR-24-CE40-2216 STOIQUES.
D.H.'s work was funded 
by FONDECYT 1231401, 
by
Centro de Modelamiento Matem\'atico (CMM) BASAL fund FB210005 for Center of Excellence from ANID,
and by the project ``Análisis de la sismicidad inducida en minería subterránea'' from CODELCO.
    
V.R's work was supported by Centro de Modelamiento Matem\'atico (CMM) BASAL fund FB210005 for Center of Excellence from ANID.
   
\section*{Appendix}

\begin{lemma}\label{Lemma1}
Let $\mu$ be a positive finite Radon measure. 
Let $(f_j)$, $(g_j)$ be sequences in $L^2(\Omega,\mu)$ and in $L^\infty(\Omega,\mu)$ respectively.
Suppose that 
\[
\sup_{j \in \mathbb{N}} \|g_j\|_{L^\infty} \leq M \quad \text{for some } M > 0.
\]
In addition, assume that $g_j \to g$ $\mu$-a.e.,
and $f_j \to f$ in $L^2(\Omega,\mu)$,  
for some $(f, g) \in L^2(\Omega,\mu) \times L^\infty(\Omega,\mu)$.
Then
\[
f_j g_j \to f g \quad \text{in } L^2(\Omega,\mu).
\]
\end{lemma}

\begin{proof}
Let $\varepsilon > 0$. Since $f^2 \in L^1$, there exists $\delta > 0$ such that 
for all $\mu$-measurable set $E$,
\[
\mu(E) < \delta \Rightarrow \int_E |f|^2 \, d\mu < \frac{\varepsilon^2}{3^2 \cdot (2M)^2}.
\]

By Egorov's Theorem, there exists a measurable set $\widetilde{E}$ 
such that $\mu(\widetilde{E}) < \delta$ and $g_j \to g$ uniformly in $\Omega \setminus \widetilde{E}$. Then, there exists $J\in \mathbb{N}$ such that for all $j\geq J$,
\begin{eqnarray*}
\|g_j-g\|_{L^{\infty}(\Omega\setminus\widetilde{E})}\leq \frac{\varepsilon}{3\|f\|_{L^2(\Omega)}}, 
&\quad \textrm{and}\quad& 
\|f_j-f\|_{L^2(\Omega)}\leq \frac{\varepsilon}{3M}.
\end{eqnarray*}
We then have
\begin{eqnarray*}
\|f_j g_j - fg\|_{L^2(\Omega)} =
&=& 
\|(f_j - f)g_j + (g_j - g)f\|_{L^2(\Omega)}
\;\leq\; \|(f_j - f)g_j\|_{L^2(\Omega)} + \|(g_j - g)f\|_{L^2(\Omega)} 
\\
&=& \left( \int_{\Omega} (f_j - f)^2 g_j^2 \right)^{1/2}
+ \left( \int_{\Omega} (g_j - g)^2 f^2 \right)^{1/2}
\\
&\leq& M \|f_j - f\|_{L^2(\Omega)} + \left( \int_{\widetilde{E}} f^2 
+ \int_{\Omega \setminus \widetilde{E}} (g_j - g)^2 f^2 \right)^{1/2}
\\
&=& M \|f_j - f\|_{_{L^2(\Omega)}}
+ \left( (2M)^2 \int_{\widetilde{E}} f^2 \, d\mu
+ \left( \frac{\varepsilon}{3 \|f\|_{L^2(\Omega)}} \right)^2 \int_\Omega f^2 \, d\mu \right)^{1/2}
\\
&\leq& 
\frac{\varepsilon}{3}+\sqrt{\frac{\varepsilon^2}{3^2}+\frac{\varepsilon^2}{3^2} }
\;\leq\;  
(1 + \sqrt{2})\frac{\varepsilon}{3}\;\leq\; \varepsilon.
\end{eqnarray*}
\end{proof}


\bibliographystyle{plain}
\bibliography{Biblio}

\end{document}